\documentclass[oneside,british]{amsart}
\usepackage[T1]{fontenc}
\usepackage[latin9]{inputenc}
\usepackage{hyperref}
\usepackage{subcaption}
\usepackage{color}
\usepackage{textcomp}
\usepackage{amsthm}
\usepackage{amstext}
\usepackage{amssymb}
\usepackage{graphicx}
\usepackage{esint}
\usepackage{tikz}\usetikzlibrary{calc}
\usepackage{pgfplots}
\usetikzlibrary{arrows}
\usepackage{mathtools}
\usepackage{enumitem}
\usepackage[labelformat=simple]{subcaption}

\makeatletter
\numberwithin{equation}{section}
\numberwithin{figure}{section}

\newtheorem{thm}{Theorem}[section]
\newtheorem{lem}[thm]{Lemma}
\newtheorem{cor}[thm]{Corollary}
\newtheorem{prop}[thm]{Proposition}

\theoremstyle{definition}
\newtheorem{defn}[thm]{Definition}
\newtheorem{defnprop}[thm]{Definition and Proposition}

\theoremstyle{remark}
\newtheorem{rem}[thm]{Remark}

\numberwithin{equation}{section}
\numberwithin{figure}{section}

\newcommand{\defeq}{\coloneqq}
\newcommand{\eqdef}{\eqqcolon}
\newcommand{\eigenf}{h}
\renewcommand{\Re}{\mathfrak{Re}}
\newcommand{\vertiii}[1]{{\left\vert\kern-0.25ex\left\vert\kern-0.25ex\left\vert #1 
    \right\vert\kern-0.25ex\right\vert\kern-0.25ex\right\vert}}

\newcommand{\eps}{\varepsilon}
\renewcommand{\emptyset}{\varnothing}

\usepackage{dsfont}
\usepackage{bbm}

\newcommand{\inte}{\textup{int}}
\newcommand{\om}{\omega}
\newcommand{\ee}{\mathrm{e}}
\newcommand{\e}{\varepsilon}
\newcommand{\mdim}{\delta}

\newcommand{\bij}{\pi }
\newcommand{\bd}{\rho}

\DeclareMathOperator{\card}{card}

\newcommand{\diam}{\textup{diam}}
\newcommand{\id}{\textup{id}}
\newcommand{\omneu}{u}

\newcommand{\z}{\zeta}
\newcommand{\aaa}{a}

\DeclareMathOperator*{\wlim}{w-lim\,}

\begin{document}

\title[Minkowski measurability of infinite cGDS]{Minkowski measurability of infinite conformal graph directed systems and application to Apollonian packings}

\author{Marc Kesseb\"ohmer}  
\address{Universit\"at Bremen, FB03 -- Mathematik und Informatik, Bibliothekstr.\ 1, 28359 Bremen, Germany}
\email{mhk@math.uni-bremen.de}
\author{Sabrina Kombrink}
\address{ Universit\"at zu L\"ubeck,
Institut f\"ur Mathematik,
Ratzeburger Allee 160,
23562 L\"ubeck, Germany}
\email{kombrink@math.uni-luebeck.de}

\begin{abstract}
We give conditions for the existence of the Minkowski content of limit sets stemming from infinite conformal graph directed systems. 
As an application we obtain Minkowski measurability of Apollonian gaskets, provide explicit formulae of the Minkowski content, and prove the analytic dependence on the initial circles. 
Further, we are able to link the fractal Euler characteristic, as well as the Minkowski content, of Apollonian gaskets with the asymptotic behaviour of the circle counting function studied by Kontorovich and Oh. These results lead to a new interpretation and an alternative formula for the Apollonian constant. We  estimate a first lower bound for the Apollonian constant, namely $0{.}055$, partially answering an open problem by Oh of 2013.  In the higher dimensional  setting of collections of disjoint balls, generated e.\,g.\ by Kleinian groups of Schottky type, we prove that all fractal curvature measures exist and are constant multiples of each other.
Further number theoretical applications connected to the Gauss map and to the Riemann $\zeta$-function illustrate our results.
\end{abstract}
\subjclass[2010]{28A75, 52C26, 28A80}
\keywords{Minkowski content; Minkowski measurability; fractal Euler characteristic; Apollonian packing, infinite conformal graph directed system, Apollonian constant, fractal curvature measures}

\maketitle

\section{Introduction}

Limit sets of infinite conformal graph directed systems (cGDS) form a striking class of geometric objects, which comprises the classes of limit sets of Kleinian groups, Apollonian circle packings, self-conformal  and self-similar sets. 

Our main motivation for the present article is to characterise the geometric structure of limit sets of infinite cGDS beyond their ``fractal dimensions''. The class of Apollonian circle packings illustrates well why this is of particular interest: 
The Hausdorff, packing and Minkowski dimension of any Apollonian circle packing is the same; its numerical value being approximately $1{.}30568\ldots$, see \cite{MR658230,MR1637951}. Thus, finer characteristics are required for distinguishing between different circle packings. For this, we
study their Minkowski content (which can be viewed as ``fractal volume''),
their surface area based content (``fractal boundary length'') and
their fractal Euler characteristic, as well as the respective localised versions, which are finite Borel measures, called the fractal curvature measures (see Sec.~\ref{sec:fcm}).
We discover that all three characteristics, as well as the respective measures, exist and are constant multiples of one another (see Thm.~\ref{thm:Apollo}, Cor.~\ref{cor:SABC}, Thm.~\ref{thm:Euler}, and  Rem.~\ref{rem:Apollofcm}) with the constant being independent of the underlying circle packing.
Even more significant, we show that the fractal curvature measures of collections of disjoint balls in arbitrary dimension formed e.\,g.\ by Kleinian groups of Schottky type are constant multiples of each other with the constant being independent of the collection of balls (see Sec.~\ref{sec:Apollohigher}). 
The fractal Euler characteristic in essence provides an asymptotic on the number $R(\varepsilon)$ of balls  of radius bigger than $\varepsilon$ as $\varepsilon\searrow 0$. In the context of circle packings in $\mathbb R^2$ the circle-counting function $R$ has been studied with the help of Laplace eigenfunctions by Kontorovich and Oh in \cite{MR2784325} (for further references see Sec.~\ref{sec:Oh}) and its asymptotic bahaviour has been derived. 
More precisely, it has been shown that $\lim_{\e\to 0}\e^D R(\e)=\pi^{-D/2}\cdot c_A\cdot\mathcal H^{D}(F)$, where $c_A$ is a universal constant which does not depend on the circle packing and its residual set $F$. Here $\mathcal H^D(F)$ denotes the $D$-dimensional Hausdorff-measure of $F$ with $D$ denoting the Hausdorff-dimension. To compute (or estimate) $c_A$ is formulated as an open problem in \cite{MR3220891}.
Our result here (Thm.~\ref{thm:Euler}) provides a different representation of the leading asymptotic term of $R$ and in this way gives a new geometric interpretation of the Apollonian constant. Through this new formula, we obtain a first lower bound for the Apollonian constant, namely we show $c_A\geq 0{.}055$ (Thm.~\ref{thm:Appoconst}).

Apollonian circle packings do not fall into any category of sets, for which the Minkowski content or the fractal curvature measures have been shown to exist and determined. The reason, why they can now be treated lies in the key novelty of the present article, namely the extension from finitely to infinitely generated systems. The most important tools in the proof of our main results are some recently obtained renewal theorems for subshifts of finite type over an infinite alphabet developed by the authors in \cite{KK15a}. 

For general limit sets of infinite cGDS, we focus on the Minkowski content. Under certain regularity conditions, we show that the cGDS being non-lattice implies existence of the Minkowski content (see Thm.~\ref{thm:main}) and provide a formula. In the lattice case, we obtain a periodic oscillating function for the volume of the $\varepsilon$-parallel set which we also present. In addition to the Apollonian circle packings, we apply our results for general limit sets of infinite cGDS to number theoretically relevant sets, more precisely to sets stemming from restricted continued fraction digits and from restricted L\"uroth digits.

The article is organised as follows. In Sec.~\ref{sec:defns} we provide the basic definitions which we need to present our main results for infinite cGDS in Sec.~\ref{sec:main}. The following section, Sec.~\ref{sec:application}, is devoted to applications of the results from Sec.~\ref{sec:main} and to examples. Here, major focus lies on Apollonian circle packings (Sec.~\ref{sec:Apollonian}), Apollonian sphere packings in $\mathbb R^3$, and collections of disjoint balls in $\mathbb R^d$ with $d\geq 4$ formed e.g. by Kleinian groups of Schottky type (Sec.~\ref{sec:Apollohigher}). 
 Connections of our results to circle-counting results by Kontorovich and Oh, which lead to estimates of the Apollonian constant, are stated in Sec.~\ref{sec:Oh}.
The number theoretical examples are given in Sec.~\ref{sec:CF} and \ref{sec:Lueroth}. In Sec.~\ref{sec:prelim} we present the preliminaries for the proofs which are provided in  the final  Sec.~\ref{sec:proofs}.

\section{Basic definitions}\label{sec:defns}

\subsection{Conformal graph directed systems}

Let $(V,E,i,t)$ be a directed multigraph with finite vertex set $V$, countable (finite or infinite) set of directed edges and functions $i,t\colon E\to V$ which determine the initial and terminal vertex of an edge. An $(E\times E)$- matrix $A=(A_{e,e'})_{e,e'\in E}$ with entries in $\{0,1\}$ which satisfies $A_{e,e'}=1$ if and only if $t(e)=i(e')$ for edges $e,e'\in E$ is called an \emph{incidence matrix}.
The set of \emph{infinite admissible words} given by $A$ is defined to be
\begin{equation}\label{eq:EAinfty}
	E^{\infty}\coloneqq E_A^{\infty}\coloneqq \left\{\om=(\om_1,\om_2,\cdots)\in E^{\mathbb N}\mid A_{\om_n,\om_{n+1}}=1\ \forall\ n\in \mathbb N\right\}.
\end{equation}
The set of sub-words of length $n\in\mathbb N$ is denoted by $E_A^n$ and the set of all finite sub-words including the empty word $\varnothing$ by $E^*$.
The incidence matrix $A$ is said to be \emph{finitely irreducible} if there exists a finite set $\Lambda\subset E^{*}$ such that for all $i,j\in E$ there is an $\omega\in\Lambda$ with $i\omega j\in E^{*}$.

\begin{defn}[GDS]\label{def:GDS}
	A \emph{graph directed system (GDS)} consists of a directed multigraph $(V,E,i,t)$ with incidence matrix $A$, a family  $(X_v)_{v\in V}$ of non-empty compact connected subsets of the Euclidean space $(\mathbb{R}^{d},\lvert\cdot\rvert)$  and for each edge $e\in E$ an injective contraction $\phi_e\colon X_{t(e)}\to X_{i(e)}$ with Lipschitz constant less than or equal to $r$ for some $r\in (0,1)$. Briefly, the family $\Phi\defeq(\phi_e\colon X_{t(e)}\to X_{i(e)})_{e\in E}$ is called a GDS.
	
A GDS is called \emph{conformal (cGDS)} if 
	\begin{enumerate}[labelwidth=\widthof{(cGDS-4)},label=(cGDS-\arabic*),leftmargin=!]
		\item for every vertex $v\in V$, $X_v$ is a compact connected set satisfying $X_v=\overline{U_v}$ with $U_v\defeq\inte(X_v)$ denoting the interior of $X_v$,
		\item\label{it:cGDS:OSC} the \emph{open set condition (OSC)} is satisfied, in that, for all $e\neq e'\in E$ we have $\phi_{e}(U_{t(e)})\cap \phi_{e'}(U_{t(e')})=\varnothing,$ 
		\item for every vertex $v\in V$ there exists an open connected set $W_v\supset X_v$ such that for every $e\in E$ with $t(e)=v$ the map $\phi_e$ extends to a conformal diffeomorphism from $W_v$ into $W_{i(e)}$, whose derivative $\phi'_e$ is bounded away from zero on $W_v$, and
		\item the \emph{cone condition} holds, that is there exist $j,\ell>0$ with $j<\pi/2$ such that for every $x\in\bigcup_{v\in V} X_v$ there exists an open cone $\textup{Con}(x,j,\ell)\subset\inte(\bigcup_{v\in V} X_v)$ with vertex $x$, central angle of measure $j$ and altitude $\ell$.
\end{enumerate}
	A cGDS, whose maps $\phi_e$ are {\em similarities} for $e\in E$, is referred to as {\em sGDS}.
\end{defn}

For a finite word $\om\in E^*$ we let $n(\om)$ denote its length, where $n(\varnothing)\defeq 0$, define $\phi_{\varnothing}$ to be the identity map on the union $X\defeq\bigcup_{v\in V}X_v$, and for $\om\in E^*\setminus\{\varnothing\}$ set
\[
	\phi_{\om}\defeq\phi_{\om_1}\circ\cdots\circ\phi_{\om_{n(\om)}}\colon X_{t(\om_{n(\om)})}\to X_{i(\om_1)}.
\]
Here $\om_i$ denotes the $i$-th letter of the word $\om$ for $i\in\{1,\ldots,n(\om)\}$, that is $\om=(\om_1,\ldots,\om_{n(\om)})$.
For two finite words $\omneu=(\omneu_1,\ldots,\omneu_n)$, $\om=(\om_1,\ldots,\om_m)\in E^{*}$ with $t(u_n)=i(\om_1)$, we let $\omneu\om\defeq(\omneu_1,\ldots,\omneu_n,\om_1,\ldots,\om_m)\in E^*$ denote their concatenation. Likewise, we set $\omneu\om\defeq(\omneu_1,\ldots,\omneu_n,\om_1,\om_2,\ldots)$ if $\om=(\om_1,\om_2,\ldots)\in E^{\infty}$.
We write $\om\vert_n\coloneqq (\om_1,\ldots,\om_n)$ for the initial sub-word of length $n\in\mathbb N$ of $\om\in E^{\infty}\cup\bigcup_{k>n}E_A^k$.

For $\om\in E^{\infty}$ the sets $(\phi_{\om\vert_n}(X_{t(\om_n)}))_{n\in\mathbb N}$ form a descending sequence of non-empty compact sets and therefore $\bigcap_{n\in \mathbb N} \phi_{\om\vert_n}(X_{t(\om_n)})\neq\varnothing$. Recall from Definition \ref{def:GDS} that $r\in(0,1)$ denotes a common Lipschitz constant of the functions $\phi_e$, $e\in E$. Since $\diam(\phi_{\om\vert_n}(X_{t(\om_n)}))\leq r^n\diam(X_{t(\om_n)})\leq r^n\max\{\diam(X_v)\mid v\in V\}$  for every $n\in\mathbb N$, the intersection 
\[
	\bigcap_{n\in\mathbb N}\phi_{\om\vert_n}(X_{t(\om_n)})
\]
is a singleton and we denote its only element by $\bij(\om)$. The map $\bij\colon E^{\infty}\to\bigcup_{v\in V} X_v$ is called the \emph{code map}. 

\begin{defn}[Limit set of a cGDS]
	The \emph{limit set} of the cGDS $(\phi_e)_{e\in E}$ is defined to be 
	\[
		F\defeq\bij(E^{\infty}).
	\]
\end{defn}
Limit sets of cGDS often have a fractal structure. They include invariant sets of conformal iterated function systems.
	A \emph{conformal iterated function system (cIFS)} is a cGDS $\Psi\defeq(\psi_1,\ldots,\psi_N)$ whose set of vertices $V$ is a singleton and whose set of edges contains $N\in\mathbb N\setminus\{1\}$ elements.
	The unique limit set of a cIFS is called the \emph{self-conformal set} associated with $\Psi$.
	In the case that the maps $\psi_1,\ldots,\psi_N$ are similarities, the limit set is called the \emph{self-similar set} associated with $\Psi$ and $\Psi$ is called an \emph{sIFS}. A core text concerning cGDS is \cite{1033.37025}.

\subsection{Dimensions, Minkowski content, fractal curvature measures, and Bowen's formula for limit sets of cGDS }\label{sec:fcm}
In order to gain a better understanding of the geometry of a limit set $F$ of a cGDS we study the volume of its \emph{$\eps$-neighbourhoods}
\[
F_{\eps}\defeq\left\{x\in\mathbb R^d\mid \inf_{y\in F} |x-y|\leq\eps\right\}
\]
for  $\eps>0$, which is a well-used approach in fractal geometry.
We measure the volume with the $d$-dimensional Lebesgue measure $\lambda_d$ and study the limiting behaviour of $\lambda_d(F_{\eps})$ as $\eps\downarrow0$. Related to this limiting behaviour is the Minkowski dimension of $F$  \[\dim_M(F)\defeq \lim_{\e \downarrow 0} d-\ln(\lambda_d(F_{\e}))/\ln(\e).\] It is given by Bowen's Formula, which we present in the following, after introducing some central thermodynamical tools.

The space $E^{\infty}$ gives rise to a topological dynamical system with the dynamics given by the \emph{(left) shift map} $\sigma$ acting on $E^{\infty}\cup E^{*}$ by 
\[
\sigma(\omega)=\begin{cases}
(\omega_{2},\omega_{3},\ldots) & \colon\omega=(\omega_{1},\omega_{2},\ldots)\in E^{\infty}\\
(\omega_{2},\omega_{3},\ldots,\omega_{n}) & \colon\omega=(\omega_{1},\omega_{2},\ldots,\omega_{n})\in E_A^{n},\ n\geq2\\
\emptyset & \colon\omega\in E_A^{0}\cup E_A^{1}.
\end{cases}
\]
For $\omega\in E_A^{n}$ we define its \emph{$n$-cylinder} to be the set
\[
\left[\omega\right]\defeq\left\{ x\in E^{\infty}\mid \forall i\in\left\{ 1,\ldots,n\right\} :x_{i}=\omega_{i}\right\} .
\]
We equip $E^{\mathbb N}$ with the product topology of the discrete topologies on $E$ and equip $E^{\infty}\subset E^{\mathbb N}$ with the subspace topology. By $\mathcal C(E^{\infty})$   we denote the set of continuous  real-valued functions on $E^{\infty}$ and call elements of $\mathcal C(E^{\infty})$ \emph{potential functions}. 
The set of bounded continuous functions in $\mathcal C(E^{\infty})$   with respect to the supremum-norm $\|\cdot\|_{\infty}$ is denoted by $\mathcal C_{b}(E^{\infty})$.
A potential function $f$ is called \emph{co-homologous} to a potential function $\zeta$ if there exists $\psi\in\mathcal C(E^{\infty})$ such that $f=\zeta+\psi-\psi\circ\sigma$. The potential function $f$ is called \emph{lattice}, if it is co-homologous to a potential function $\zeta$ whose range is contained in a discrete subgroup of $\mathbb R$. Otherwise we say that $f$ is \emph{non-lattice}.

A central role in our studies is played by the \emph{geometric potential function} $\xi\colon E^{\infty}\to\mathbb R$ associated with a cGDS $\Phi$, which is defined by $\xi(\om)\defeq -\ln\lvert\phi'_{\om_1}(\bij(\sigma\om))\rvert$ for $\om=(\om_1,\om_2,\ldots)\in E^{\infty}$. It lies in $\mathcal C(E^{\infty})$ but is generally unbounded if $\Phi$ is infinitely generated. 
We call $\Phi$ \emph{(non-)lattice} if its geometric potential function is (non-)lattice.
The \emph{topological pressure function} of a potential function  $f\in\mathcal C(E^{\infty})$  with respect to the shift map $\sigma\colon E^{\infty}\to E^{\infty}$ is defined by the well-defined limit 
\[
P(f)\defeq\lim_{n\to\infty}\frac{1}{n}\log\sum_{\omega\in E_A^{n}}\exp\left(\sup_{\tau\in[\omega]}S_{n}f(\tau)\right),
\]
where
\[
S_{n}f\defeq\sum_{j=0}^{n-1}f\circ\sigma^{j}\ \text{for}\ n\geq1\quad\text{and}\quad S_{0}f\defeq0
\]
denotes the \emph{$n$-th Birkhoff sum} of $f$.

Bowen's Formula \cite{1033.37025} states that the Hausdorff dimension $\dim_H(F)$ and the Minkowski dimension $\dim_M(F)$ of the limit set $F$ coincide and are given by 
\[
\dim_{H}(F)=\dim_{M}(F)=D\defeq\inf\{s>0:P(-s \xi )\leq 0 \}.
\]
The system $\Phi$ is called \emph{regular} if $P(-D\xi)=0$.
It is called {\em strongly regular} if
\begin{equation}\label{eq:theta}
\theta\defeq \sup\left\{s\in \mathbb{R}\mid \sum_{e\in E}\exp\left(\sup(s \xi \vert_{[e]})\right)<\infty \right\}>-D.
\end{equation}
On the region where $P(s\xi)$ is finite, the map $s\mapsto P(s\xi)$ is continuous. Therefore, strong regularity implies regularity.

There exists a big variety of limit sets of cGDS of the same Minkowski and Hausdorff dimension, whence finer tools are needed and we propose to study the \emph{Minkowski content}
 \[
 \mathcal M(F)\defeq\lim_{\eps\to 0}\eps^{D-d}\lambda_d(F_{\eps})
 \]
of $F$. 
A set $A$ for which $\mathcal M(A)$ exists, is positive and finite is called \emph{Minkowski measurable}.
A refinement is provided by the \emph{local Minkowski content} $ \mathcal M(F,B)$ relative to a Borel set $B\subset\mathbb R^d$ 
 if it exists:
\[
 \mathcal M(F,B)\defeq\lim_{\eps\to 0}\eps^{D-d}\lambda_d(F_{\eps}\cap B).
 \]

Related to the Minkowski content are the fractal curvatures and fractal curvature measures as introduced in \cite{Winter}.
These are defined via a (local) Steiner formula for sets of positive reach. A set $A\subset\mathbb R^d$ is said to be of \emph{positive reach} if there exists $r>0$ such that any point $x$ in $A_r$ has a unique closest neighbour $\pi_A(x)$ in $A$. The supremum $\textup{reach}(A)$ over all such $r>0$ is called the \emph{reach} of $A$, and $\pi_A\colon\inte A_{\textup{reach}(A)}\to A$ is called the \emph{metric projection} onto $A$.
\begin{thm}[{Local Steiner formula, \cite{Federer}}] 
 Let $A\subset\mathbb R^d$ be a compact set of positive reach. Then there exist uniquely determined signed Borel measures $C_0(A,\cdot),\ldots, C_d(A,\cdot)$ such that for every $B\in\mathcal B(\mathbb R^d)$ and every $0\leq\varepsilon<\textup{reach}(A)$ we have
 \[
  \lambda_d(A_{\e}\cap\pi_A^{-1}(B))
  =\sum_{k=0}^d\e^{d-k}\kappa_{d-k}C_k(A,B),
 \]
  where $\kappa_k$ denotes the $k$-dimensional volume of the $k$-dimensional unit ball.
\end{thm}
The signed measure $C_k(A,\cdot)$ is called the \emph{$k$-th curvature measure} of $A$.
$C_k(A)\defeq C_k(A,\mathbb R^d)$ is called the \emph{$k$-th (total) curvature} of $A$. For any set $A$ of positive reach, $C_d(A)=\lambda_d(A)$, $C_{d-1}(A)=\lambda_{d-1}(\partial A)/2$ and $C_0(A)=\chi(A)$, where $\chi$ denotes the Euler characteristic.

Unfortunately, fractal sets are not of positive reach. However, for such sets one can consider their parallel sets. If these are of positive reach (which will be the case in the applications of the present paper) it is of interest to consider the asymptotic behaviour of the curvature measures as the parallel set becomes smaller. For the following we assume that $F_{\ee^{-t}}$ is a set of positive reach for all sufficiently large $t\in\mathbb R$ and denote by $D$ the Minkowski dimension of $F$.
Whenever, the weak limit
\[
 C_k^f(F,\cdot)\defeq\wlim_{t\to\infty}\ee^{t(k-D)} C_k(F_{\ee^{-t}},\cdot)
\]
exists, it is called the \emph{$k$-th fractal curvature measure} of $F$. Its existence implies the existence of the \emph{$k$-th (total) fractal curvature} of $F$, which is defined through
\[
 C_k^f(F)\defeq\lim_{t\to\infty}\ee^{t(k-D)} C_k(F_{\ee^{-t}}).
\]
Notice, for $k=d$ we obtain the Minkowski content, for $k=d-1$ the \emph{surface area based content}, and for $k=0$ the \emph{fractal Euler characteristic}.

\section{Main results}\label{sec:main}

Throughout, we let $\Phi$ denote a strongly regular cGDS with limit set $F$ and use the notation from Sec.~\ref{sec:defns}. 
A non-empty open set $O\defeq\bigcup_{v\in V} O_v$ with $O_v\subseteq U_v$ and $\phi_e O_{t(e)}\subseteq O_{i(e)}$ for each $e\in E$ is called a \emph{feasible open set} for the cGDS $\Phi$.
The condition \ref{it:cGDS:OSC} implies existence of at least one feasible open set. 
For convenience we write  $\phi_{\om}O\defeq\phi_{\om}O_{t(\om)}$ as well as $\phi_{\om}F\defeq\phi_{\om}(F\cap X_{t(\om)})$ for $\om \in E^*$. Here $i(\omega)\defeq i(\omega_1)$ for $\omega\in E^*\cup E^{\infty}$ and $t(\omega)\defeq t(\omega_n)$ for a finite word $\omega\in E_A^n$.
Moreover, we assume the following projection condition.
\begin{defn}\label{def:metproj}
     The cGDS $\Phi$ together with the feasible open set $O$ is said to satisfy the \emph{projection condition} if
     \begin{align*}\label{eq:PC}
       \phi_e O \subseteq \overline{\pi_{\overline{F}}^{-1}(\overline{\phi_e F})}\quad\text{for}\ e\in E.
     \end{align*}
 Here, $\pi_{\overline{F}}$ denotes the metric projection onto the topological closure $\overline{F}$ of $F$.
\end{defn}
It is straight-forward to see that the projection condition implies
\begin{equation}\label{eq:projcond}
   F_{\eps}\cap \phi_eO=(\phi_e F)_{\eps}\cap \phi_e O
\end{equation}
for each $\eps>0$ and $e\in E$.

\begin{rem}\label{lem:centralOSC}
 Any cGDS $\Phi$ satisfies the projection condition together with its \emph{central open set} $O\defeq\bigcup_{v\in V}O_v$, where 
  \[
   O_v\defeq\inte\left(\bigcap_{\om\in E^*:t(\om)=v}\phi_{\om}^{-1}\overline{\pi_{\overline{F}}^{-1}\overline{\phi_{\om}F}} \right),
  \]
 whenever $O\neq\emptyset$.
 Note that when it is non-empty $O$ indeed is a feasible open set for $\Phi$: For all $\om\in E^*$ with $t(\om)=i(e)$ we have $\om e\in E^*$ and $t(\om e)=t(e)$ implying  
 \[
   \phi_{\om}\phi_e O_{t(e)}
   \subset \phi_{\om}\phi_e\inte\left(\phi_{\om e}^{-1}\overline{\pi_{\overline F}^{-1}\overline{\phi_{\om e}F}}\right)
   \subset \inte\left(\overline{\pi^{-1}_{\overline F}\overline{\phi_{\om}F}}\right)
 \]
 and whence $\phi_e O_{t(e)}\subset O_{i(e)}$.\\   
 The central open set was introduced in  \cite{MR2199182} for finite self-similar systems and proven to be non-empty for such systems. Thus, in this case, the projection condition is not a restriction on $\Phi$ or $F$ but rather ensures a convenient choice for a feasible open set.
 The above definition provides an extension to infinitely generated cGDS.
\end{rem}

We write
\begin{equation}\label{eq:Gamma}
 \Gamma_v\defeq O_v\setminus \bigcup_{e\in E:i(e)=v}\phi_e(O_{t(e)}),   \qquad \Gamma\defeq\bigcup_{v\in V} \Gamma_v
\end{equation}
and assume \emph{non-triviality}, that is $\lambda_d(\Gamma)>0$. 
\begin{prop}\label{prop:non-trivial}
	Non-triviality implies $\dim_M(F)<d$ and thus $\lambda_d(F)=0$. 
	\end{prop}
In Sec.~\ref{sec:proofs} we prove the above proposition, which is an extension of  \cite[Prop.~4{.}5{.}9 and Thm.~4{.}5{.}10]{1033.37025}.

Our main result requires the following conditions
\begin{enumerate}[label=(\Alph*)]
	\item\label{it:cond1} The projection and non-triviality conditions are satisfied.
	\item\label{it:cond2} The incidence matrix is finitely irreducible.
	\item\label{it:cond3} $\Phi$ is strongly regular.  
	\item\label{it:cond4} There exists $c,\gamma>0$ for which $\lambda_d (F_{\e}\cap\Gamma)\leq c \e^{d-D+\gamma}$. 
\end{enumerate}
\begin{rem}
	If $E$ is of finite cardinality then conditions \ref{it:cond3} and \ref{it:cond4} are always satisfied. 
\end{rem}

Two functions $f,g\colon\mathbb R\to\mathbb R$ are called \emph{asymptotic} as $t\to\infty$, written $f(t)\sim g(t)$ as $t\to\infty$, if for all $\e>0$ there exists $\widetilde{t}\in\mathbb R$ such that for all $t\geq \widetilde{t}$ the value $f(t)$ lies between $(1-\e)g(t)$ and $(1+\e)g(t)$. 
For $t\in\mathbb R$ we define $\lfloor t\rfloor\defeq\max\{k\in\mathbb Z\mid k\leq t\}$ and $\{t\}\defeq t-\lfloor t\rfloor\in[0,1)$. Note that $\lfloor t\rfloor$ and $\{t\}$ respectively are the {\em integer} and the {\em fractional part} of $t\geq 0$.

\begin{thm}\label{thm:main}
 Suppose that conditions \ref{it:cond1} to \ref{it:cond4} are met. Let $\mu_{-D\xi}$ denote the unique $\sigma$-invariant Gibbs state for $-D\xi$ and let $\nu$ be the $D$-conformal measure associated with $\Phi$ (see Sec.~\ref{sec:thermo}).  For any Borel subset $B$ of $O$ the following hold:
 \begin{enumerate}
  \item If $\xi$ is non-lattice then
    \begin{align*}
      &\lambda_d\left(F_{\ee^{-t}}\cap B\right)\\
      &\sim  \frac{\ee^{-t(d-D)}}{\int \xi\,\textup{d}\mu_{-D\xi}} \lim_{m\to\infty}
      \sum_{u\in E_A^m}\int^{\infty}_{-\infty}
      \ee^{-T(D-d)}\lambda_d\left(F_{\ee^{-T}}\cap\phi_u\Gamma_{t(u)}\right)\,\textup{d}T\cdot\nu(B).
     \end{align*}
     Thus, the local Minkowski content $\mathcal M(F,B)$ relative to $B$ exists.
   \item If $\xi$ is lattice with $\xi=\zeta+\psi-\psi\circ\sigma$ and if $\aaa>0$ denotes the maximal real for which $\zeta(E^{\infty})\subset\aaa\mathbb Z$  then
     \begin{align*}
      &\lambda_d\left(F_{\ee^{-t}}\cap B\right)\\
      &\sim  \frac{\aaa \cdot \ee^{-t(d-D)}}{\int \xi\,\textup{d}\mu_{-D\xi}} \lim_{m\to\infty}
      \sum_{u\in E_A^m}\ee^{(D-d)\psi(x_u)}\sum_{\ell=-\infty}^{\infty}
      \textup{e}^{-\aaa\ell(D-d)}\\
      &\qquad\times\int_{E^{\infty}}\mathds 1_{B}(\pi y)\cdot\ee^{-\aaa (D-d)\left\{\frac{t+\psi(x_u)-\psi(y)}{\aaa}\right\}} A_{u,\ell}(t,y)\,\textup{d}\nu_{-D\xi}(y)
     \end{align*}  
     for arbitrary $x_u\in[u]$, where 
     \[
      A_{u,\ell}(t,y)\defeq \lambda_d\left(F_{\ee^{-\left(\aaa\ell+\aaa\left\{\frac{t+\psi(x_u)-\psi(y)}{\aaa}\right\}\right)}}\cap\phi_{u}\Gamma_{t(u)}\right).
     \]
 \end{enumerate}
\end{thm}

Immediate consequences of Thm.~\ref{thm:main} are presented in the following corollaries, which we state without proofs. 

\begin{cor}
	If the conditions of   Thm.~\ref{thm:main} are met, then
	\[
	\lambda_d(F_{\ee^{-t}}\cap B) \sim \lambda_d(F_{\ee^{-t}}\cap O) \cdot \nu(B).
	\]
	 In particular, in the non-lattice situation, $B\mapsto \lim_{t\to\infty}\ee^{-t(D-d)}\lambda_d(F_{\ee^{-t}}\cap B)$ is a constant multiple of the $D$-conformal measure.
\end{cor}

\begin{cor}
	Additionally to the conditions of Thm.~\ref{thm:main}, suppose that $d=1$. Let $\Gamma=\bigcup_k\Gamma_k$ be the decomposition of $\Gamma$ into its connected components.  If $\xi$ is non-lattice then
    \begin{align*}
      \mathcal M(F,B)=
      \frac{2^{1-D}}{D(1-D)\int \xi\,\textup{d}\mu_{-D\xi}}\lim_{m\to\infty}
      \sum_{u\in E_A^m}\sum_k
    \left\lvert\phi_u\Gamma_k\right\rvert^D\cdot\nu(B).
     \end{align*}
\end{cor}
\begin{cor}
	Additionally to the conditions of Thm.~\ref{thm:main}, suppose that $\Phi$ is an sGDS and let $r_i$ denote the similarity ratio of $\phi_i$. Then the following hold.
	 \begin{enumerate}
  \item If $\xi$ is non-lattice then
    \begin{align*}
      & \mathcal M(F,B)\\
      &=  \frac{1}{\int \xi\,\textup{d}\mu_{-D\xi}} \lim_{m\to\infty}
      \sum_{u\in E_A^m} r_u^D
      \int^{\infty}_{-\infty}
      \ee^{-T(D-d)}\lambda_d\left(F_{\ee^{-T}}\cap\Gamma_{t(u)}\right)\,\textup{d}T\cdot\nu(B).
     \end{align*}
   \item If $\xi$ is lattice and $\aaa>0$ denotes the maximal real for which $\xi(E^{\infty})\subset\aaa\mathbb Z$  then
     \begin{align*}
      &\lambda_d\left(F_{\ee^{-t}}\cap B\right)\\
      &\sim \lim_{m\to\infty}\!
      \sum_{u\in E_A^m} r_u^D\!\sum_{\ell=-\infty}^{\infty}\!\!
      \ee^{-\aaa\ell(D-d)}\ee^{-\aaa (D-d)\left\{\frac{t}{\aaa}\right\}} \lambda_d\left(F_{\ee^{-(\aaa\ell+\aaa\left\{\frac{t}{\aaa}\right\})}}\cap\Gamma_{t(u)}\right)\\
      &\quad\times\frac{\aaa  \cdot \ee^{-t(d-D)} }{\int \xi\,\textup{d}\mu_{-D\xi}}\cdot\nu(B).
     \end{align*}
 \end{enumerate}
\end{cor}
\begin{cor}
	Additionally to the conditions of Thm.~\ref{thm:main}, suppose that $\Phi$ is an sIFS and let $r_i$ denote the similarity ratio of $\phi_i$. Then the following hold.
	 \begin{enumerate}
  \item If $\xi$ is non-lattice then
    \begin{align*}
      \mathcal M(F,B)
      =  \frac{1}{-\sum_{i}r_i^D \ln r_i } 
      \int^{\infty}_{-\infty}
      \ee^{-T(D-d)}\lambda_d\left(F_{\ee^{-T}}\cap\Gamma\right)\,\textup{d}T\cdot\nu(B).
     \end{align*}
   \item If $\xi$ is lattice and $\aaa>0$ denotes the maximal real for which $\xi(E^{\infty})\subset\aaa\mathbb Z$  then
     \begin{align*}
      &\lambda_d\left(F_{\ee^{-t}}\cap  B\right)\\
      &\sim  \frac{\aaa \,\ee^{-t(d-D)}}{-\sum_{i} r_i^D\ln r_i } \sum_{\ell=-\infty}^{\infty}
      \textup{e}^{-\aaa\ell(D-d)}\ee^{-\aaa (D-d)\left\{\frac{t}{\aaa}\right\}} \lambda_d\left(F_{\ee^{-(\aaa\ell+\aaa\left\{\frac{t}{\aaa}\right\})}}\cap\Gamma\right)\cdot\nu(B).
     \end{align*}
 \end{enumerate}
\end{cor}
\begin{rem}
	The above theorem and corollaries provide analogues of the respective theorems in the case that the alphabet is finite, given in \cite{Diss,KK12,Kocak1,KK15,Winter2015285,renewal}. 
\end{rem}

\section{Applications and examples}\label{sec:application}

\subsection{Apollonian gaskets}\label{sec:Apollonian}
The study of circle packings has a long history. 
A result by Apollonius (ca. 262--190 BC), that is of main importance to us is the following, see \cite[Thm.~1.1]{Pollicott_Apo}. Given three mutually tangent circles $C_1,C_2,C_3$ with disjoint interiors there are precisely two circles $C_0,C_4$ which are tangent to each of the original three (see Fig.~\ref{fig:tangentcircles}). Now, for each triple of mutually tangent circles from the collection $\{C_0,\ldots,C_4\}$, we can again find two circles which are tangent to each of the circles from the triple. For instance the circles $C_3$ and $C_7$ are tangent to the circles $C_0,C_1,C_2$ (see Fig.~\ref{fig:tangentcircles5}). In this way we obtain a packing of the circle $C_0$ generated by $C_1,C_2,C_3$.
We call the limiting object an \emph{Apollonian circle packing}.

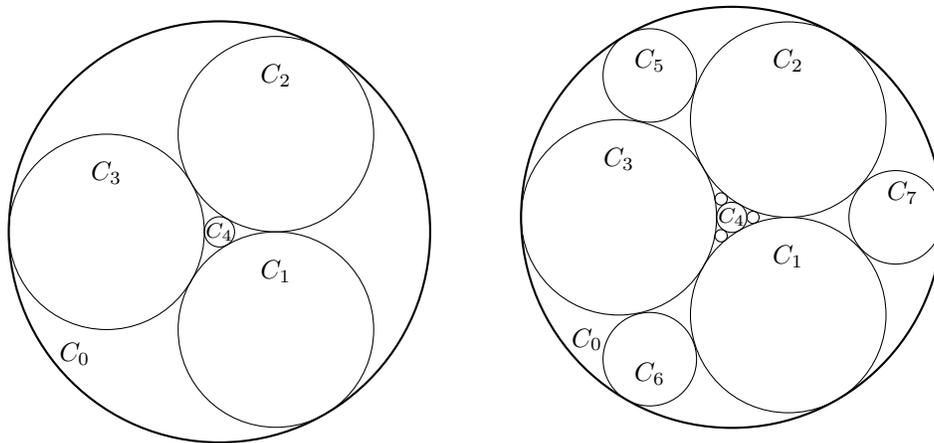
\begin{figure}[b]\centering
\begin{subfigure}{0.45\textwidth}
\begin{tikzpicture}[scale=0.65]
	\draw[thick] ({sqrt(3)*4/3},0) circle({sqrt(3)*4/3+2});
	\draw   (0,0) circle(2);
	\draw	({2*sqrt(3)},2) circle(2);
	\draw   (0,1.2) node {$C_3$};
	\draw   ({2*sqrt(3)},3.2) node {$C_2$};
	\draw   ({2*sqrt(3)},-0.8) node {$C_1$};
	\draw	({2*sqrt(3)},-2) circle(2);
	\draw   ({sqrt(3)*4/3},0) circle({sqrt(3)*4/3-2});
	\draw   (-0.65,-2.5) node {$C_0$};
	\draw   ({sqrt(3)*4/3},0) node {\footnotesize$C_4$};
\end{tikzpicture} \caption{Three mutually tangent circles $C_1$, $C_2$, $C_3$ and the circles $C_0$, $C_4$ which are tangent to each of the original three.} 
\label{fig:tangentcircles}
\end{subfigure}\hspace{1cm}
\begin{subfigure}{0.45\textwidth}
\begin{tikzpicture}[scale=0.65]
	\draw[thick] ({sqrt(3)*4/3},0) circle({sqrt(3)*4/3+2});
	\draw   (0,0) circle(2);
	\draw	({2*sqrt(3)},2) circle(2);
	\draw   (0,1.2) node {$C_3$};
	\draw   ({2*sqrt(3)},3.2) node {$C_2$};
	\draw   ({2*sqrt(3)},-0.8) node {$C_1$};
	\draw	({2*sqrt(3)},-2) circle(2);
	\draw   ({sqrt(3)*4/3},0) circle({sqrt(3)*4/3-2});
	\draw   (-0.65,-2.5) node {$C_0$};
	\draw   ({sqrt(3)*4/3},0) node {\footnotesize$C_4$};
	\draw	({sqrt(3)*8/3+2-(32+8*sqrt(3))/(2*sqrt(3)+12)/3.1},0) circle({(32+8*sqrt(3))/(2*sqrt(3)+12)/3.1});
	\draw   (5.8,0.5) node {$C_7$};
	\draw	({4/3*sqrt(3)-(sqrt(3)*4/3+2-(32+8*sqrt(3))/(2*sqrt(3)+12)/3.1)/2},{(sqrt(3)*4/3+2-(32+8*sqrt(3))/(2*sqrt(3)+12)/3.1)/2*sqrt(3)}) circle({(32+8*sqrt(3))/(2*sqrt(3)+12)/3.1});
	\draw   ({4/3*sqrt(3)-(sqrt(3)*4/3+2-(32+8*sqrt(3))/(2*sqrt(3)+12)/3.1)/2},3.2) node {$C_5$};
	\draw	({4/3*sqrt(3)-(sqrt(3)*4/3+2-(32+8*sqrt(3))/(2*sqrt(3)+12)/3.1)/2},{-(sqrt(3)*4/3+2-(32+8*sqrt(3))/(2*sqrt(3)+12)/3.1)/2*sqrt(3)}) circle({(32+8*sqrt(3))/(2*sqrt(3)+12)/3.1});
	\draw   ({4/3*sqrt(3)-(sqrt(3)*4/3+2-(32+8*sqrt(3))/(2*sqrt(3)+12)/3.1)/2},-3.2) node {$C_6$};
	\draw   ({80/33*sqrt(3)-48/33},0) circle({18/33-8/33*sqrt(3)});
	\draw   ({26/33*sqrt(3)+8/11},{18/11-8/11*sqrt(3)}) circle({18/33-8/33*sqrt(3)});
	\draw   ({26/33*sqrt(3)+8/11},{-18/11+8/11*sqrt(3)}) circle({18/33-8/33*sqrt(3)});
\end{tikzpicture} \caption{For each triple of mutually tangent circles from $\{C_0,\ldots,C_4\}$ there are exactly two circles which are tangent to the circles from the triple.}
\label{fig:tangentcircles5}
\end{subfigure}
\caption{Families of tangent circles generated by $C_1$, $C_2$, $C_3$.}
\end{figure}
\begin{figure}\centering
\begin{subfigure}{0.43\textwidth}
\includegraphics[width=\textwidth]{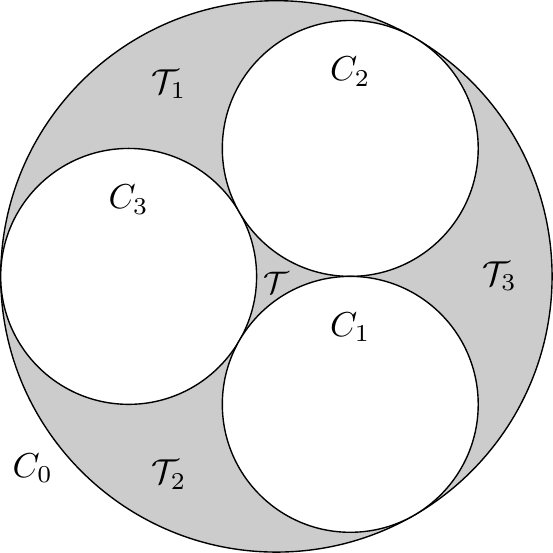}
\caption{The four curvlinear triangles $\mathcal T$, $\mathcal T_1$, $\mathcal T_2$, $\mathcal T_3$ external to $C_1$, $C_2$, $C_3$ and internal to $C_0$.} \label{fig:curvlinear}
\end{subfigure}\hspace{1cm}
\begin{subfigure}{0.43\textwidth}
\centering
\includegraphics[width=\textwidth]{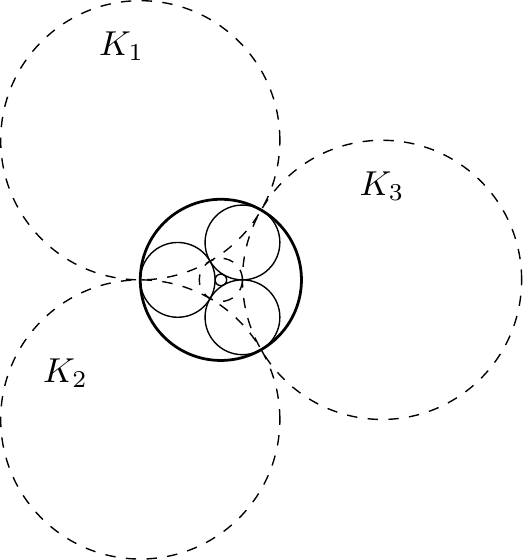}
\caption{Dual circles $K_0$, $K_1$, $K_2$ and $K_3$ to $C_0$, $C_1$, $C_2$ and $C_3$.} 
\label{fig:dual}
\end{subfigure}
\begin{subfigure}{0.43\textwidth}
\centering
\includegraphics[width=\textwidth]{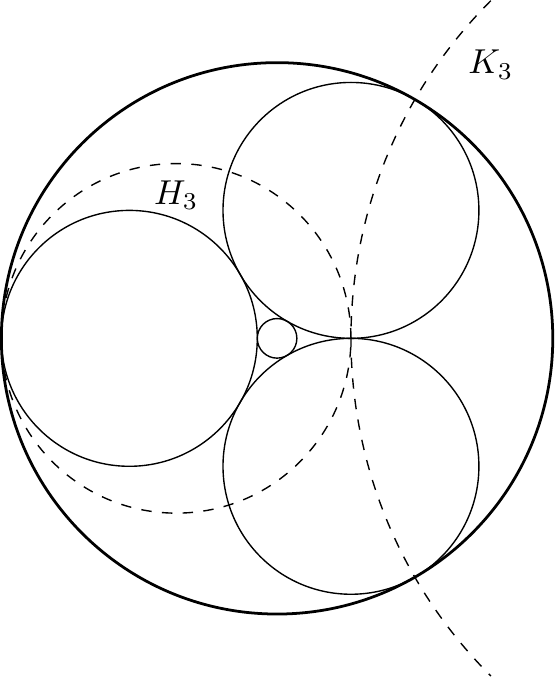}
\caption{Reflection circles $K_3$ and $H_3$.} 
\label{fig:horo}
\end{subfigure}\hspace{1cm}
\begin{subfigure}{0.43\textwidth}
\centering
\includegraphics[width=\textwidth]{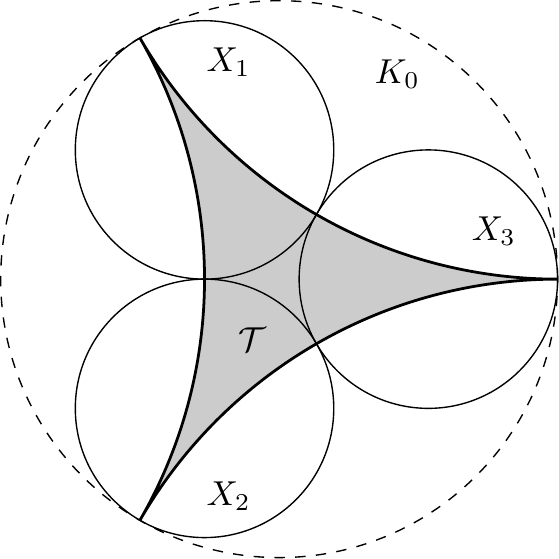}
\caption{The compact connected sets $X_1$, $X_2$ and $X_3$ of the cGDS.} 
\label{fig:Xv}
\end{subfigure}

\begin{subfigure}{0.43\textwidth}
\centering
\includegraphics[width=\textwidth]{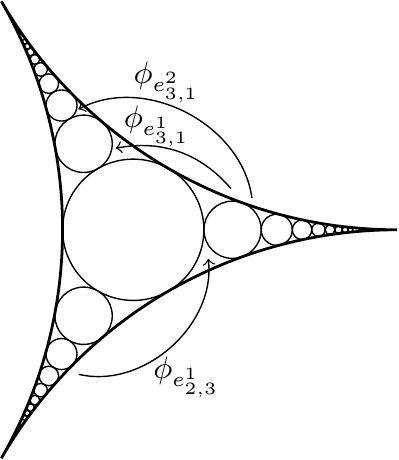}
\caption{The action of the cGDS.} 
\label{fig:The-Apollonian-Packing}
\end{subfigure}
\hspace{1cm}
\begin{subfigure}{0.43\textwidth}
\centering
\includegraphics[width=\textwidth]{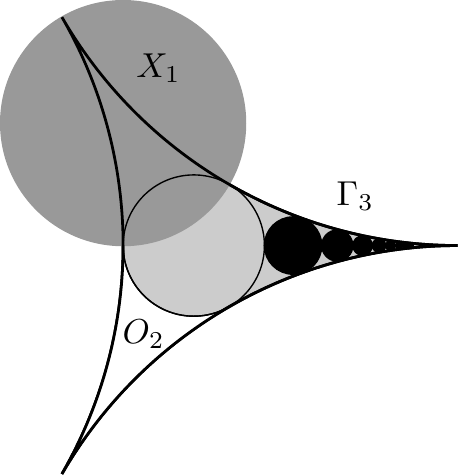}
\caption{The sets $X_1$ (dark grey), $O_2$ (white) and $\Gamma_3$ (black).}
\label{fig:Gamma}
\end{subfigure}
\caption{Defining the cGDS for an Apollonian circle packing.}
\label{fig:Apollo}
\end{figure}

We are interested in geometric properties of circle packings.  
For this we represent the circle packing inside each of the four curvlinear triangles external to $C_1,C_2,C_3$ and internal to $C_0$ (denoted by $\mathcal T$, $\mathcal T_1$, $\mathcal T_2$, $\mathcal T_3$ in Fig.~\ref{fig:curvlinear}) as a limit set of an infinite cGDS. Without loss of generality we focus on the central curvlinear triangle $\mathcal T$, which is bounded by arcs of the circles $C_1,C_2,C_3$. 
(Note that the packing inside of $\mathcal T_j$, $j\in\{1,2,3\}$, is the image of the packing of $\mathcal T$ under a M\"obius transformation. In Thm.~\ref{thm:image} for instance, we will see that the Minkowski content of the packing inside of $\mathcal T_j$ can thus be deduced from the Minkowski content of the packing inside of $\mathcal T$.)
In order to define the contractions of the cGDS we introduce dual- and horocircles:
If four circles touch each other mutually, another set of four circles of mutual contact can be found whose points of contact coincide with those of the first four. The new four circles are called \emph{dual circles}.  We denote the dual circles of $C_0,C_1,C_2,C_3$ by $K_0,K_1,K_2,K_3$  (see Fig.~\ref{fig:dual}).
Moreover, for $j\in\{1,2,3\}$ we let $H_j$ denote the circle which tangentally touches $C_0$ and $C_j$ at their touching point and which goes through the touching point of $C_{i_1}$ and $C_{i_2}$, where $i_1,i_2\in\{1,2,3\}\setminus \{j\}$. We call $H_j$ the \emph{horocircle} associated with $C_j$ (see Fig.~\ref{fig:horo}).

For a circle $C$ we let $R_C$ denote the reflection on $C$. If $m$ denotes the centre and $r$ the radius of $C$ then 
\[
 R_C\colon\mathbb C\setminus\{m\}\to \mathbb C\setminus\{m\},\qquad
 R_C(z)= \frac{r^2(z-m)}{\lvert z-m \rvert^2}+m.
\]
For obtaining the contractions of the cGDS associated with $\mathcal T$ we introduce isometries of $\mathbb{H}^{2}$ the three dimensional hyperbolic space by
\[
f_{j}\defeq R_{H_j}\circ R_{K_j}:K_0\to K_0
\]
for $j\in\{1,2,3\}$. 
With $m,n$ denoting the centres and $r,s$ denoting the radii of $H_j$ and $K_j$ respectively an explicit representation of $f_j$ is given by
\[
 f_j(z)
 = \frac{z(r^2+m(\overline{n-m}))-nr^2+ms^2-mn(\overline{n-m})}{z(\overline{n-m})+s^2-n(\overline{n-m})}.
\]
Here, $\overline{n}$ denotes the complex conjugate of $n\in\mathbb C$.

We define three compact connected sets $X_{v}\defeq f_{v}\left(K_0\right)$
in $\partial\mathbb{H}^2=\mathbb{R}^{2}$ for $v\in V\defeq\{1,2,3\}$ (see Fig.~\ref{fig:Xv}). 
For $v,w\in V$ with $v\neq w$ we let $E_{v,w}\defeq\{e_{v,w}^k\mid k\in\mathbb N\}$ denote a directed set of edges with $t(e_{v,w}^k)=v$, $i(e_{v,w}^k)=w$ and associated contractions 
\[
 \phi_{e_{v,w}^k}\defeq f_w^k\vert_{X_v}\colon X_v\to X_w
\]
(see Fig.~\ref{fig:The-Apollonian-Packing}).
Let $E\defeq\bigcup_{v\in V}\bigcup_{w\in V\setminus\{v\}} E_{v,w}$ and define an incidence matrix $A=(A_{e,e'})_{e,e'\in E}$ by $A_{e_{v,w}^k, e_{v',w'}^{\ell}}=\mathds 1_{\{w'=v\}}$.
 Then $\Phi\defeq \left\{ \phi_e\mid e\in E \right\}$ defines a cGDS whose limit set is the Apollonian circle packing generated by $C_1,C_2,C_3$ inside of $\mathcal T$.

We now check the assumptions of Sec.~\ref{sec:main} in order to apply Thm.~\ref{thm:main}. 
We set $O_v\defeq\inte(f_v(\mathcal T))$. Then the set $\Gamma_v$ is a union of countably many circles: $\Gamma_v=\bigcup_{k=1}^{\infty} f_v^k(C_4)$. For our depicted example the sets $X_1$, $O_2$ and $\Gamma_3$ are shown in Fig.~\ref{fig:Gamma}.
Moreover, $O=\bigcup O_v$ is a feasible open set for $\Phi$ for which the projection and non-triviality conditions are clearly satisfied giving \ref{it:cond1}.
Setting $\Lambda\defeq\{\emptyset, e_{1,2}^1,e_{1,3}^1,e_{2,1}^1,e_{2,3}^1,e_{3,1}^1,e_{3,2}^1\}$ shows finite irreducibility of the incidence matrix and thus \ref{it:cond2}.
To show \ref{it:cond3}, that is strong regularity, we use the next lemma.

\begin{lem}\cite{MR1623671}\label{lem:k-2}
There is a constant $Q>1$ such that for all $v\in\left\{ 1,2,3\right\} $ we have
\[
Q^{-1}k^{-2}
\leq \sup_{x\in\mathcal T\setminus {X_v}}\left\lvert (f_v^k)'(x)\right\rvert
\leq Q k^{-2}.
\]
\end{lem}
\begin{rem}
 In \cite{MR1623671} the bounds in the above lemma are stated for $\|\phi'_{e_{v,w}^k}\|_{\infty}$ instead of $\sup_{x\in\mathcal T\setminus {X_v}}\left\lvert (f_v^k)'(x)\right\rvert$. However, the discrepancy vanishes in the constant $Q$.
\end{rem}
Since $\xi\geq 0$ and $E$ is of infinite cardinality, the expression $\sum_{e\in E}\exp(\sup(s\xi\vert_{[e]}))$ is infinite for $s\geq 0$. To determine $\theta$ from \eqref{eq:theta} it thus suffices to consider $s<0$. For such $s$
\begin{align*}
 \sum_{e\in E}\exp\sup(s\xi\vert_{[e]})
 =\sum_{v=1}^3\sum_{k=1}^{\infty}\exp\left(-s\ln\sup\lvert (f_v^k)'\rvert_{\mathcal T}\right)
 \in\left[3Q^s\sum_{k=1}^{\infty}k^{2s},3Q^{-s}\sum_{k=1}^{\infty}k^{2s}\right]
\end{align*}
by Lem.~\ref{lem:k-2}. We deduce $\theta=-1/2$.
McMullen  \cite{MR658230,MR1637951} determined  the Hausdorff, packing and Minkowski dimensions of the Apollonian gasket to be $D=1{.}30568...$.
Therefore, $\theta>-D$ and strong regularity follows.
To verify the last condition \ref{it:cond4} from Sec.~\ref{sec:main} we study $\lambda_d(F_{\e}\cap \Gamma_v)$ for each $v\in V$. 
Let $r$ denote the radius of $C_4$. Then by Lem.~\ref{lem:k-2} the radius of $f_j^k(C_4)$ is bounded from above by $rQk^{-2}$. Thus,
\begin{align*}
	\lambda_2(F_{\e}\cap \Gamma_v)
	&\leq \sum_{k=1}^{\left\lfloor \sqrt{\frac{rQ}{\e}} \right\rfloor}\pi(2rQk^{-2}\e-\e^2)
	+ \sum_{k=\left\lfloor \sqrt{\frac{rQ}{\e}} \right\rfloor+1}^{\infty} \pi (r Q k^{-2})^2\\
	&\leq C\cdot \e 
	=C\cdot \e^{2-D+(D-1)},
\end{align*}
with some constant $C>0$. Therefore, \ref{it:cond4} is satisfied with $\gamma=D-1>0$.

The following lemma is probably well known to the experts. Here, we provide a short proof for completeness.

\begin{lem} 
 The cGDS associated with any Apollonian circle packing is non-lattice.
\end{lem}
\begin{proof}
  Let $\Phi$ and $\Psi$ denote cGDS associated with different circle packings and let $\xi$ and $\zeta$ denote their respective geometric potential functions. Then there exists a M\"obius transformation $g$ such that $\Psi=g\circ\Phi\circ g^{-1}$ and we have $\zeta=\xi-\ln\lvert g'\circ\pi\rvert+\ln\lvert g'\circ\pi\circ\sigma\rvert$. Since $-\ln\lvert g'\circ\pi\rvert \in\mathcal C(E^{\infty})$ it follows that $\Phi$ is lattice if and only if $\Psi$ is lattice. Thus, it suffices to consider one particular circle packing. We choose the Ford-circles (see Fig.~\ref{fig:Ford}).
  \begin{figure}
   \begin{tikzpicture}[thick, scale=2.2]
    \draw (-2,1) -- (2,1);
    \draw (-2,-1) -- (2,-1);
    \draw (-1,0)  circle(1);
    \draw (1,0)  circle(1);
    \draw  (0,1/4-1) circle(1/4);   
    \draw  (1/3,1/9-1) circle(1/9);   
    \draw  (1/2,1/16-1) circle(1/16);   
    \draw  (3/5,1/25-1) circle(1/25);   
    \draw  (2/3,1/36-1) circle(1/36);
    \draw  (-1/3,1/9-1) circle(1/9);   
    \draw  (-1/2,1/16-1) circle(1/16);   
    \draw  (-3/5,1/25-1) circle(1/25);   
    \draw  (-2/3,1/36-1) circle(1/36);
    \draw  (0,7/12-1) circle(1/12); 
    \draw  (0,-1/4-1/24) circle(1/24);
    \draw  (0,-1/5-1/40) circle(1/40); 
    \draw  (1/5,1/25-1) circle(1/25);
    \draw  (-1/5,1/25-1) circle(1/25);
    \draw  (0,1-1/4) circle(1/4);
    \draw  (1/3,-1/9+1) circle(1/9);   
    \draw  (1/2,-1/16+1) circle(1/16);   
    \draw  (3/5,-1/25+1) circle(1/25);   
    \draw  (2/3,-1/36+1) circle(1/36);
    \draw  (-1/3,-1/9+1) circle(1/9);   
    \draw  (-1/2,-1/16+1) circle(1/16);   
    \draw  (-3/5,-1/25+1) circle(1/25);   
    \draw  (-2/3,-1/36+1) circle(1/36);
    \draw  (0,-7/12+1) circle(1/12); 
    \draw  (0,1/4+1/24) circle(1/24);
    \draw  (0,1/5+1/40) circle(1/40);   
    \draw  (1/5,-1/25+1) circle(1/25); 
    \draw  (-1/5,-1/25+1) circle(1/25);
    \draw  (-2,-1+1/4) circle(1/4);
    \draw  (1/3-2,1/9-1) circle(1/9);   
    \draw  (1/2-2,1/16-1) circle(1/16);   
    \draw  (3/5-2,1/25-1) circle(1/25);   
    \draw  (2/3-2,1/36-1) circle(1/36);
    \draw  (-2,7/12-1) circle(1/12); 
    \draw  (-2,-1/4-1/24) circle(1/24);
    \draw  (-2,-1/5-1/40) circle(1/40); 
    \draw  (1/5-2,1/25-1) circle(1/25); 
    \draw  (-1/5+2,1/25-1) circle(1/25);     
    \draw  (-2,1-1/4) circle(1/4);
    \draw  (1/3-2,-1/9+1) circle(1/9);   
    \draw  (1/2-2,-1/16+1) circle(1/16);   
    \draw  (3/5-2,-1/25+1) circle(1/25);   
    \draw  (2/3-2,-1/36+1) circle(1/36);
    \draw  (-2,-7/12+1) circle(1/12); 
    \draw  (-2,1/4+1/24) circle(1/24);
    \draw  (-2,1/5+1/40) circle(1/40); 
    \draw  (1/5-2,-1/25+1) circle(1/25);
    \draw  (-1/5+2,-1/25+1) circle(1/25);    
    \draw  (2,-1+1/4) circle(1/4);
    \draw  (-1/3+2,1/9-1) circle(1/9);   
    \draw  (-1/2+2,1/16-1) circle(1/16);   
    \draw  (-3/5+2,1/25-1) circle(1/25);   
    \draw  (-2/3+2,1/36-1) circle(1/36);
    \draw  (+2,7/12-1) circle(1/12); 
    \draw  (+2,-1/4-1/24) circle(1/24);
    \draw  (+2,-1/5-1/40) circle(1/40); 
    \draw  (1/5+2,1/25-1) circle(1/25); 
    \draw  (-1/5-2,1/25-1) circle(1/25);     
    \draw  (+2,1-1/4) circle(1/4);
    \draw  (-1/3+2,-1/9+1) circle(1/9);   
    \draw  (-1/2+2,-1/16+1) circle(1/16);   
    \draw  (-3/5+2,-1/25+1) circle(1/25);   
    \draw  (-2/3+2,-1/36+1) circle(1/36);
    \draw  (2,-7/12+1) circle(1/12); 
    \draw  (2,1/4+1/24) circle(1/24);
    \draw  (2,1/5+1/40) circle(1/40);
    \draw  (1/5+2,-1/25+1) circle(1/25);
    \draw  (-1/5-2,-1/25+1) circle(1/25);
    \draw[white,fill=white] (-2.5,-1) rectangle (-2,1.1); 
    \draw[white,fill=white] (2,-1) rectangle (2.5,1.1); 
   \end{tikzpicture}
   \caption{Ford-circles.}
   \label{fig:Ford}
  \end{figure}
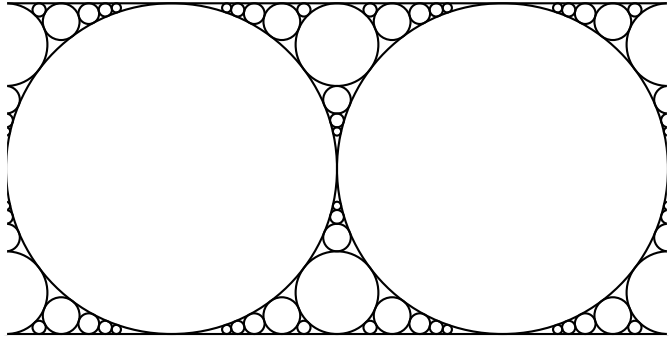
 Here, the cGDS is given through the maps
 \begin{align*}
  f_1(z)=\frac{-4}{z+3+\mathbf i}+1-\mathbf i,\quad
  f_2(z)=\frac{4}{-z+3-\mathbf i}-1-\mathbf i,\quad
  f_3(z)=\frac{1}{z-\mathbf i}-\mathbf i.
 \end{align*}
 Suppose that the associated geometric potential function $\xi$ is lattice. Then there exist $\zeta,\psi\in\mathcal C(E^{\infty})$ with $\zeta(E^{\infty})\subset\aaa\mathbb Z$ for some $\aaa>0$ and $\xi=\zeta+\psi-\psi\circ\sigma$. This implies
 \begin{equation*}
  S_2\xi=S_2\zeta+\psi-\psi\circ\sigma^2
 \end{equation*}
 and whence $S_2\xi(x)\in\aaa\mathbb Z$ for 2-periodic words $x$. Consider the following points given by two-periodic words
 \begin{align*}
  &\pi\left(\overline{e_{2,1}^1e_{1,2}^1}\right)=2-\sqrt{5}-\mathbf i,\ \ \qquad \pi\left(\overline{e_{2,1}^1e_{1,2}^2}\right)=1-\mathbf i-\frac{2}{3}\sqrt{3}\\
  & \pi\left(\overline{e_{1,2}^1e_{2,1}^1}\right)=-2+\sqrt{5}-\mathbf i,\qquad \pi\left(\overline{e_{1,2}^2e_{2,1}^1}\right)=-3-\mathbf i+2\sqrt{3}.
 \end{align*}
 We have $f_1'(\pi(\overline{e_{1,2}^1e_{2,1}^1}))=f_2'(\pi(\overline{e_{2,1}^1e_{1,2}^1}))=4/(1+\sqrt{5})^2$, $f_1'(\pi(\overline{e_{1,2}^2e_{2,1}^1}))=1/3$ and $(f_2^2)'(\pi(\overline{e_{2,1}^1e_{1,2}^2}))=3/(\sqrt{3}+2)^2$ and conclude
 \[
  \frac{S_2\xi(\overline{e_{2,1}^1e_{1,2}^1})}{S_2\xi(\overline{e_{1,2}^1e_{2,1}^1})}
  =\frac{-4\ln(2)+4\ln\left(1+\sqrt{5}\right)}{2\ln(\sqrt{3}+2)}
  \notin\mathbb Q,
 \]
 which is a contradiction.
\end{proof}

Let $\mu_{-D\xi}$ denote the unique $\sigma$-invariant Gibbs state of $-D\xi$.
Applying Thm.~\ref{thm:main} yields for any $B\in\mathcal B(\mathcal T)$
\begin{align*}
&\lambda_2(F_{\ee^{-t}}\cap B)\\
&\sim\frac{ \ee^{-t(2-D)}}{\int\xi\,\textup{d}\mu_{-D\xi}}\lim_{m\to\infty}\sum_{\om\in E_A^m}
\int_{-\infty}^{\infty}\ee^{-T(D-2)}\lambda_2(F_{\ee^{-T}}\cap\bigcup_{k=1}^{\infty}\phi_{\om} f_{t(\om)}^k (C_{4}))\,\textup{d}T\cdot\nu(B).
\end{align*}
Using that $\phi_{\om}f_{t(\om)}^k (C_4)$ are circles, the above integral can be evaluated and we obtain the following.
\begin{thm}[Apollonian gasket -- Minkowski content]\label{thm:Apollo}
The local Minkowski content exists and we have for any Borel set $B\subset \mathcal T$
\[
\mathcal{M}(F,B)
=\frac{2}{D(2-D)(D-1)}
\frac { \pi 2^{-D}}{\int \xi d\mu_{-D\xi}}
\lim_{m\to\infty}\sum_{\om\in E_A^m}\sum_{k=1}^{\infty}\left|\phi_{\om} f_{t(\om)}^k\left(C_{4}\right)\right|^{D}\cdot \nu(B),
\]
where $\nu$ denotes the $D$-conformal measure associated with $F$.
\end{thm}
\begin{thm}[Apollonian gasket -- analytic dependence]\label{thm:image}
Let $F_{0}$ denote the symmetric Apollonian gasket with corners on the unit circle and $\nu$ the associated $D$-conformal measure.  Then for any other Apollonian gasket $F$ there exists a unique M\"obius transformation $g$ such that $F=g(F_{0})$ and for the Minkowski content we have for any $B\in\mathcal B(\mathcal T_0)$
\[
\frac{\mathcal{M}(F,g(B))}{\mathcal{M}(F_{0},B)}=\int \vert g' \vert^{D} \,\textup{d}\nu.
\]
This in particular proves the analytic dependence of $\mathcal{M}(F)$
on the initial circles. 
\end{thm}

The proof of Thm.~\ref{thm:image} is given in Sec.~\ref{sec:proofs}.

\begin{cor}[Apollonian gasket -- surface area based content]\label{cor:SABC}
The first fractal curvature measure, i.\,e.\ the \emph{surface area based content}, exists and we have for any $B\in\mathcal B(\mathcal T)$
\[
\mathcal{C}_1^f(F,B)
=\frac{\pi}{D(D-1)}\frac {2^{-D}}{\int \xi d\mu_{-D\xi}}
\lim_{m\to\infty}\sum_{\om\in E_A^m}\sum_{k=1}^{\infty}\left|\phi_{\om} f_{t(\om)}^k\left(C_{4}\right)\right|^{D}\cdot\nu(B).
\]
\end{cor}
The preceding corollary follows from a result in \cite{MR2865426} together with Thm.~\ref{thm:Apollo}.
\begin{thm}[Apollonian gasket -- fractal Euler characteristic]\label{thm:Euler}
	The  $0$-th fractal curvature measure, i.\,e.\ the \emph{localised fractal Euler characteristic}, exists and we have for any $B\in\mathcal B(\mathcal T)$
\[
\mathcal{C}_0^f(F,B)
=-\frac{1}{D}\frac {2^{-D}}{\int \xi d\mu_{-D\xi}}
\lim_{m\to\infty}\sum_{\om\in E_A^m}\sum_{k=1}^{\infty}\left|\phi_{\om} f_{t(\om)}^k\left(C_{4}\right)\right|^{D}\cdot\nu(B).
\]
\end{thm}
The proof of Thm.~\ref{thm:Euler} is given in Sec.~\ref{sec:proofs}.
\begin{rem}\label{rem:Apollofcm}
Combining Thm.~\ref{thm:Apollo}, Cor.~\ref{cor:SABC}, and Thm.~\ref{thm:Euler} we see that  
\[
\mathcal C_0^f(F,\cdot)=\frac{1}{\kappa_2}(1-D)\mathcal C_1^f(F,\cdot)\quad\text{and}\quad
\mathcal C_1^f(F,\cdot)=\frac{1}{\kappa_1}(2-D)\mathcal C_2^f(F,\cdot)
\]
with $\kappa_k$ denoting the $k$-dimensional volume of the $k$-dimensional unit ball. Thus, the Minkowski content, the surface area based content, and the fractal Euler characteristic of Apollonian circle packings are all constant multiples of one another with the constant being independent of the underlying circle packing.  
What is more, this result holds even for the respective measures.
That the Minkowski content and the surface area based content are constant multiples of each other is precisely the statement of \cite{MR2865426}  which applies to any bounded set. However, this relation between the fractal Euler characteristic and the Minkowski content is not known for general sets and might be specific to circle packings. 
This observation shows that the surface area based content and the Euler characteristic do not provide any further geometric information on the structure of the underlying Apollonian circle packing in addition to the Minkowski content.
\end{rem}

\subsection{Circle counting}\label{sec:Oh}
	The fractal Euler characteristic in essence gives an asymp\-totic on the number $R(\varepsilon)$ of circles in $\mathcal T$ of radius bigger than $\varepsilon$ as $\varepsilon\searrow 0$. More precisely,
	\[
	R(\varepsilon)\sim- \varepsilon^{-D}\mathcal C_0^f(F,\mathcal T)
	\]
	as  $\varepsilon\searrow 0$. The \emph{circle counting function} $R$ has been studied with the help of Laplace eigenfunctions by Kontorovich and Oh in \cite{MR2784325} (see also \cite{MR2827842,MR2784325, MR3053757,MR3173439,MR3220891,MR3497263,Pollicott_Apo}) and its asymptotic bahaviour has been derived. Our result here provides a different representation of the constant factor of  the leading asymptotic term and in this way gives a new geometric interpretation.

\begin{defnprop}[The Apollonian constant, {\cite{MR3220891}}]\label{defnprop:constant}
	Let $F$ denote a circle packing with associated circle counting function $R$ and let $\mathcal H^D$ denote the $D$-dimensional Hausdorff measure.
	 \[
	 c_A\defeq \frac{\pi^{D/2}}{\mathcal H^D(F)}\cdot \lim_{\e\to0} \e^D R(\e)
	 \]
	is called the \emph{Apollonian constant}. It is a universal constant, which is independent of the circle packing.
	What is more, for any set $B\in\mathcal B(\mathcal T)$
	\[
	\mathcal C_0^f(F,B)
	=- c_A\pi^{-D/2}\mathcal H^D(F\cap B).
	\]
\end{defnprop}

An immediate consequence of the above definition and Thm.~\ref{thm:Euler} is the following.
\begin{cor}\label{cor:Apolloconstant} We have that 
	\[
	c_A= \frac {2^{-D}\pi^{D/2}}{D\int \xi d\mu_{-D\xi}}\cdot
\frac{\lim_{m\to\infty}\sum_{\om\in E_A^m}\sum_{k=1}^{\infty}\left|\phi_{\om} f_{t(\om)}^k\left(C_{4}\right)\right|^{D}}
{\mathcal H^D(F)},
	\]
	where each of the two fractions are constants that are independent of the particular circle packing.
\end{cor}
Moreover, the above corollary immediately implies that $\mathcal H^D$ is a constant multiple of the $D$-conformal measure $\nu$, reproducing a fact that is well known for the Patterson measure.

\begin{thm}\label{thm:Appoconst}
	Let $F_{0}$ denote the symmetric Apollonian gasket with corners on the unit circle as depicted in Fig.~\ref{fig:Apollo}. Further, let $q\colon \mathbb C\to\mathbb C$ be the M\"obius transform
	\[
	q(z)\defeq \frac{(1+(1+\mathbf i)\sqrt{3})z-1}{z-1+(1+\mathbf i)\sqrt{3}},
	\]
	which maps the real line to the circle $C_2$, so that $C_4$ is mapped to $\partial X_1$. 
	Then 
	\[
	c_A
	\geq\frac{ 2^{-D}\pi^{D/2}}{D\cdot\int\xi\,\textup{d}\mu_{-D\xi}}\cdot
	\frac{\lim_{m\to\infty}\sum_{\om\in E_A^m}\sum_{k=1}^{\infty}\left|\phi_{\om} f_{t(\om)}^k \left(C_{4}\right)\right|^{D}}{6\lim_{m\to\infty}\sum_{\om\in E_A^m,t(\om)=3}\sum_{k=1}^{\infty}\left|\phi_{\om} f_{3}^k q\left(C_{4}\right)\right|^{D}}
	\geq 0{.}055.
	\]
\end{thm}

The proof of  Thm.~\ref{thm:Appoconst} is presented in Sec.~\ref{sec:proofs}.

\subsection{Apollonian sphere packings in {$\mathbb R^3$} and collections of balls in higher dimensions}\label{sec:Apollohigher}
In case of Apollonian sphere packings in $\mathbb R^3$ it is interesting to consider not only the Minkowski content, the surface area based content and the fractal Euler characteristic, but also the other fractal curvatures and curvature measures. 

In analogy to the 2-dimensional setting, we can construct a cGDS which generates the sphere packing in a region $\mathcal T$ (see \cite{MR2183490}). When considering higher dimensions, i.\,e.\ $\mathbb R^d$ with $d\geq 4$, we look at collections of disjoint  balls formed e.\,g.\ by Kleinian groups of Schottky type. These have a representation as a limit set of a cGDS, see \cite[Ch.~5]{1033.37025}. In the following we use the same notation as in Sec.~\ref{sec:Apollonian}.
We see that $F_{\ee^{-t}}$ is a set of positive reach for any $t>0$. Through the local Steiner formula we immediately obtain for $\ee^{-t}<\lvert\phi_u\Gamma_{t(u)}\rvert$
\begin{align}
 C_d(F_{\ee^{-t}},\phi_u\Gamma_{t(u)})
 &=\kappa_d\left(\frac{\lvert\phi_u\Gamma_{t(u)}\rvert}{2}\right)^d-\kappa_d\left(\frac{\lvert\phi_u\Gamma_{t(u)}\rvert}{2}-\ee^{-t}\right)^d\quad\text{and}\nonumber\\
 C_k(F_{\ee^{-t}},\phi_u\Gamma_{t(u)})
 &=\frac{\kappa_d}{\kappa_{d-k}}\binom{d}{k}(-1)^{d-k+1}\left(\frac{\lvert\phi_u\Gamma_{t(u)}\rvert}{2}-\ee^{-t}\right)^k,\quad k\leq d-1. \label{eq:Ck}
\end{align}
We apply the same methods that we used in the previous sections. Note, that in general it is not possible to apply the bounded distortion lemma, when considering curvature measures, since $C_k$ is not monotonic in the first component for $k\leq d-1$. However, here we may apply bounded distortion to the expressions on the right hand side of \eqref{eq:Ck} and get
\begin{align*}
      C_k\left(F_{\ee^{-t}},O\right)
      \sim  \frac{\ee^{-t(k-D)}}{\int \xi\,\textup{d}\mu_{-D\xi}} \lim_{m\to\infty}
      \sum_{u\in E_A^m}\underbrace{\int^{\infty}_{-\infty}
      \ee^{-T(D-d+k)}C_k\left(F_{\ee^{-T}},\phi_u\Gamma_{t(u)}\right)\,\textup{d}T}_{{\defeq  I_u^k}}.
\end{align*}

Evaluating the integrals yields
\begin{align*}
 I_u^d &=\frac{\kappa_d d!}{D(D-1)\cdots(D-d+1)(d-D)}\left(\frac{\lvert\phi_u\Gamma_{t(u)}\rvert}{2}\right)^D,\\
 I_u^k &= \frac{\kappa_d \cdot d!}{\kappa_{d-k}\cdot(d-k)!}(-1)^{d-k+1} \left(\frac{\lvert\phi_u\Gamma_{t(u)}\rvert}{2}\right)^D\frac{1}{D(D-1)\cdots(D-k)}, \quad k\leq d-1.
\end{align*}
Altogether, we obtain for $k\leq d-1$
\begin{align*}
 C_k^f(F,O)=\frac{\kappa_{d-k-1}}{\kappa_{d-k}}\cdot\frac{k+1-D}{d-k}\cdot C_{k+1}^f(F,O).
\end{align*}
This shows that the fractal curvatures of Apollonian sphere packings and limit sets of Kleinian groups of Schottky type are constant multiples of each other, where the constants only depend on the dimension $d$ and on the indices of the respective curvatures.
In particular, when considering sphere packings in $\mathbb R^3$ we deduce that
\begin{align*}
 C_0^f(F,O) &=\frac{1}{4} (1-D)\cdot C_1^f(F,O),\\
 C_1^f(F,O) &=\frac{1}{\pi} (2-D)\cdot C_2^f(F,O),\\
 C_2^f(F,O) &=\frac{1}{2} (3-D)\cdot C_3^f(F,O).
\end{align*}
Also in this section it is possible to localise the curvatures and to obtain the analogues results for the fractal curvature measures.

\subsection{Restricted continued fraction digits}\label{sec:CF}

Let us consider the fractal set given by a restricted continued fraction digit set. For
$\Lambda\subset \mathbb{N}$ we define
\[
F_{\Lambda}\defeq\{[a_{1},a_{2},\ldots]\mid \forall n\in \mathbb{N}; a_{n}\in \Lambda \}.
\]
The Hausdorff dimension of these sets has been extensively studied in \cite{MR2197868}. In fact, it has been shown that the Texan Conjecture holds, that is $\{\dim _{H}(F_{\Lambda})\mid \Lambda\subset \mathbb{N}\}=[0,1]$.

For these sets with  $\card(\Lambda)\geq2$ we have to consider the following conformal IFS defined on the unit interval
\[\Phi\defeq \{\phi_{k}:x\mapsto 1/(x+k)\mid k\in \Lambda\}.\]
Set  $D\defeq\dim _{H}(F_{\Lambda})$ and suppose that $D<1$ which implies $\Lambda\neq \mathbb N$. Let  $\mu$ denote the equilibrium with respect to the geometric potential $-\delta \xi$ and  $h_{\mu}$ its  measure theoretical entropy. 
For simplicity we will assume that for $k\in \mathbb{N}\setminus \Lambda$  we have $k\pm1 \in \Lambda\cup\{0\}$. This assumption guarantees that the cGDS is strongly regular (cf. \cite[Example 6.5]{MU:99}) and gives rise to the  handy formula stated next.
\begin{thm}
The set $F_{\Lambda}$ is Minkowski measurable    and   we have 
\[
\mathcal{M}(F_{\Lambda})=  \frac{2^{1-D}}{(1-D)h_{\mu}}  \lim_{m\to\infty}\sum_{a\in \mathbb{N}\setminus \Lambda} \sum_{\left|\omega\right|=m}\left|\Phi_{\omega}\left([a]\right)\right|^{D}.
\]
\end{thm}

\subsection{Restricted L\"uroth digits}\label{sec:Lueroth}

Fix the decreasing sequence $(t_{n})$ in $[0,1]$ given, for some $s>1$, by
\[
t_{n}\defeq \zeta(s)^{-1}\sum _{k=n}^{\infty} \frac{1}{k^{s}}, \; n \in \mathbb{N}
\]
defining a L\"uroth system, where $\zeta$ denotes the Riemann $\zeta$-function  (cf. \cite{MR2995653}). With $a_n\coloneqq 1/ (n^s\zeta(s))$ and the same conditions on the set $\Lambda\subset  \mathbb{N}$ as in Section \ref{sec:CF} we consider the linear IFS on the unit-interval given by  
 \[\Phi\defeq \{\phi_{n}:x\mapsto -a_n x+t_n\mid n\in \Lambda\}.\]
 
Let $\delta>0$  be the unique number such that 
\[
\zeta_\Lambda(\delta s)\coloneqq\sum _{k\in \Lambda} \frac{1}{k^{\delta s}}=\zeta(s)^{\delta}.
\]
Then the fractal set $L_{\Lambda}$ of all L\"uroth expansions omitting the digits from  $ \mathbb{N}\setminus \Lambda$ has Hausdorff and Minkowski dimension equal to $\delta$. If the system is non-lattice (depending on the particular choice of $\Lambda$ - see below) then the set is Minkowski measurable and its Minkowski content is given by
\[
\mathcal{M}(L_{\Lambda})= \frac{2^{1-\delta}(\zeta(s\delta)/\zeta(s)^{\delta}-1)}{\delta(1-\delta)\int\xi\,\textup{d}\mu_{-\delta\xi}}
= \frac{2^{1-\delta}(\zeta(s\delta)/\zeta(s)^{\delta}-1)}{(1-\delta)( \log \zeta_\Lambda(\delta s)-\delta s (\log \zeta_\Lambda)'(\delta s))}\]
 
For an example of a lattice system which is not Minkowski measurable we fix an integer  $\ell\geq 2$ and some $s>1$ for which $s\log \ell/\log \zeta(s)\in \mathbb{Q}$. Then the fractal set $L_{\Lambda}$ of all L\"uroth expansions allowing only  the digits from $\Lambda \subset \{\ell ^{k}\mid k\in \mathbb{N}\}$, $\card(\Lambda)\geq2$,   is not Minkowski measurable since $k s\log \ell - \log (\zeta (s))$ lies in the lattice $(\log(\zeta(s))/q)\mathbb{Z}$ for some $q\in \mathbb{N}$.

\section{Preliminaries for the proofs}\label{sec:prelim}
\subsection{Thermodynamic formalism}\label{sec:thermo}

A Bo\-rel probability measure $\mu$ on $E^{\infty}$ is said to be a \emph{Gibbs state} for $f\in\mathcal C(E^{\infty})$ if there exists a constant $c>0$ such that 
\begin{equation}\label{eq:Gibbs}
 c^{-1}
 \leq\frac{\mu\left([\omega\vert_n]\right)}{\exp\left(S_n f(\omega)-n P(f)\right)}
 \leq c
\end{equation}
for every $\omega\in E^{\infty}$ and $n\in\mathbb N$.

\begin{defn}[H\"older continuity]
	For $f\in\mathcal C(E^{\infty})$, $\theta\in(0,1)$ and $n\in\mathbb N$ define
	\begin{align*}
		\text{var}_n(f)&\defeq \sup\{\lvert f(x)-f(y)\rvert\mid x,y\in E^{\infty}\ \text{and}\ x_i=y_i\ \text{for}\ i\leq n\},\\
		\|f\|_{\theta}&\defeq \sup_{n\geq 1}\frac{\text{var}_n(f)}{\theta^n}\quad\text{and}\\
		\mathcal F_{\theta}(E^{\infty})&\defeq\{f\in\mathcal C(E^{\infty})\mid \|f\|_{\theta}<\infty\}.
	\end{align*}
	A function $f\in\mathcal F_{\theta}(E^{\infty})$ is called \emph{$\theta$-H\"older continuous}.  
	Since by our definition a $\theta$-H\"older continuous function is not necessarily bounded, we introduce the space of bounded H\"older continuous functions and denote it by $\mathcal F_{\theta}^b(E^{\infty})\defeq \mathcal F_{\theta}(E^{\infty}) \cap \mathcal C_b(E^{\infty})$.
\end{defn}

In order to define the central object of this section, namely the Perron-Frobenius operator of a potential function $f$, we need to assume that 
\begin{equation}\label{eq:summable}
\sum_{e\in E}\exp(\sup(f\vert_{[e]}))<\infty.
\end{equation}
A function $f\in\mathcal{F}_{\theta}(E^{\infty})$ which satisfies \eqref{eq:summable} is called \emph{summable}. Note that by \cite[Thm.~2{.}1{.}5]{1033.37025}
 the pressure   $P(u)$ is finite and $>-\infty$  for any summable H\"older continuous function $u\colon E^{\infty}\to\mathbb R$.

\begin{defn}[Perron-Frobenius operator]
Let $f\in \mathcal F_{\theta}(E^{\infty})$ be summable. 
The \emph{Perron-Frobenius-Operator} $\mathcal{L}_{f}\colon\mathcal{C}_{b}\left(E^{\infty}\right)\to\mathcal{C}_{b}\left(E^{\infty}\right)$
to the potential function $f$ acting on $\mathcal{C}_{b}\left(E^{\infty}\right)$ is defined by 
\[
\mathcal{L}_{f}(g)(\omega)
=\sum_{e\in E:A_{e\omega_{1}}=1}\textup{e}^{f(e\omega)}g(e\omega)
=\sum_{y:\sigma y=\omega}\textup{e}^{f(y)}g(y).
\]
\end{defn}
The \emph{conjugate Perron-Frobenius operator} $\mathcal{L}_{f}^{*}$ acts on  $\mathcal{C}_{b}^{*}\left(E^{\infty}\right)$ via
\[
\mathcal{L}_{f}^{*}(\mu)(g)=\mu(\mathcal{L}_{f}(g))=\int\mathcal{L}_{f}(g)\,\textup{d}\mu,
\]
as shown in \cite{KK15a}.

The following theorem is a combination of Lem.~2{.}4{.}1, Thms.~2{.}4{.}3, 2{.}4{.}6 and Cor.~2{.}7{.}5 from \cite{1033.37025}.
Note that Thms.~2{.}4{.}3, 2{.}4{.}6 in \cite{1033.37025} are stated and proved under the hypothesis that the incidence matrix $A$ is finitely primitive. In \cite{KK15a} we provided reasoning that the assumption of finitely irreducible $A$ in fact suffices.

\begin{thm}[\cite{1033.37025},  Ruelle-Perron-Frobenius theorem for infinite alphabets]\label{thm:Ruelleinfinite}
  Suppose that $f\in\mathcal F_{\theta}(E^{\infty})$ for some $\theta\in(0,1)$ is summable. Then $\mathcal L_f$ preserves the space $\mathcal F_{\theta}^b(E^{\infty})$, i.\,e.\ $\mathcal L_f\vert_{\mathcal F_{\theta}^b(E^{\infty})}\colon \mathcal F_{\theta}^b(E^{\infty})\to\mathcal F_{\theta}^b(E^{\infty})$. Moreover, the following hold.
  \begin{enumerate}
    \item  There is a unique Borel probability eigenmeasure $\nu_{f}$ of the conjugate Per\-ron-Frobenius operator $\mathcal{L}_{f}^{*}$ and the corresponding eigenvalue is equal to $\textup{e}^{P(f)}$. Moreover, $\nu_f$ is a Gibbs state for $f$.
    \item\label{it:R} The operator $\mathcal L_f\vert_{\mathcal F_{\theta}^b(E^{\infty})}$ has an eigenfunction $h_f$ which is bounded from above and which satisfies $\int h_f\,\textup{d}\nu_f=1$. Further,  there exists an $R>0$ such that $h_f\geq R$ on $E^{\infty}$.
    \item The function $f$ has a unique ergodic $\sigma$-invariant Gibbs state $\mu_{f}$.
    \item\label{it:convergencetoef} 
    There exist constants $\overline M>0$ and $\gamma\in(0,1)$ such that for every $g\in\mathcal F_{\theta}^b(E^{\infty})$ and every $n\in\mathbb N_0$
    \begin{equation}\label{eq:convergencetoef}
    \left\lVert\textup{e}^{-nP(f)}\mathcal{L}_{f}^{n}(g)-\int g\,\textup{d}\nu_{f}\cdot h_f\right\rVert_{\theta} \leq\overline{M}\gamma^{n}\left(\| g\|_{\theta}+\|g\|_{\infty}\right).
    \end{equation}
  \end{enumerate}
\end{thm}

Directly from \eqref{eq:convergencetoef} we infer the following:
\begin{cor}\label{cor:disc-real}
  In the setting of Thm.~\ref{thm:Ruelleinfinite} \ref{it:convergencetoef}, $\textup{e}^{P(u)}$ is a simple isolated eigenvalue of $\mathcal{L}_{u}\vert_{\mathcal F_{\theta}^b(E^{\infty},\mathbb R)}$. The rest of the spectrum of $\mathcal{L}_{u}\vert_{\mathcal F_{\theta}^b(E^{\infty},\mathbb R)}$  is contained in a disc centred at zero of radius at most $\gamma<\ee^{P(u)}$. 
\end{cor}

The unique probability measure $\nu$ supported on $F$ which satisfies 
\[
\nu(\phi_i X\cap \phi_j X)=0\quad\text{and}\quad \nu(\phi_i B)=\int_B\lvert \phi'_i\rvert^D\,\textup{d}\nu
\]
for all distinct $i,j\in E$ and all Borel sets $B\subset X_{t(i)}$, is called the \emph{$D$-conformal measure} associated with $\Phi$. Notice, $\nu_{-D\xi}=\nu\circ\pi$.

\subsection{Renewal theorems}\label{sec:renewal}

In \cite{Lalley} renewal theorems for counting measures in symbolic dynamics were established, where the underlying symbolic space is based on a finite alphabet. These renewal theorems were extended to more general measures in \cite{Diss,renewal}. 
Moreover, the extended versions from \cite{renewal} were generalised to the setting of an underlying  countably infinite alphabet in \cite{KK15a} and we summarise these results in the present section.

Fix $\theta\in(0,1)$ and let  $\kappa\in\mathcal F^b_{\theta}(E^{\infty})$ be non-negative but not identically zero. Further, $\xi,\eta\in\mathcal F_{\theta}(E^{\infty})$ shall satisfy the following:
\begin{enumerate}[label=(\Alph*)]\setcounter{enumi}{4}
 \item \label{it:regularPotential}\emph{Regular potential}.
  $\xi\geq 0$ is not identically zero. There exists a unique $\delta\in\mathbb R$ with $P(\eta-\delta\xi)=0$. Further, $-\delta<t^*\defeq\sup\{t\in\mathbb R\mid \eta+t\xi\ \text{is summable}\}$ and $\int -(\eta +t\xi)\,\textup{d}\mu_{\eta-\delta\xi}<\infty$ for all $t$ in a neighbourhood of $-\delta$.
\end{enumerate}
For fixed $x\in E^{\infty}$ the renewal theorem provides the asymptotic behaviour as $t\to\infty$ of the renewal function
\begin{equation}\label{eq:N}
 N(t,x)\defeq\sum_{n=0}^{\infty}\sum_{y:\sigma^ny=x}\kappa(y)f_y(t-S_n\xi(y))\textup{e}^{S_n\eta(y)},
\end{equation}
where $f_x\colon\mathbb R\to\mathbb R$, for $x\in E^{\infty}$, needs to satisfy some regularity conditions (see \ref{it:Lebesgue}--\ref{it:decay} below). We call $N$ a \emph{renewal function} since it satisfies an analogue to the classical renewal equation:
\begin{equation}\label{eq:renewaleq}
 N(t,x)=\sum_{y:\sigma y=x} N(t-\xi(y),y)\textup{e}^{\eta(y)}+\kappa(x)f_x(t).
\end{equation}

\begin{enumerate}[label=(\Alph*)]\setcounter{enumi}{5}
 \item\label{it:Lebesgue} \emph{Lebesgue integrability}. For any $x\in E^{\infty}$ the  Lebesgue integral 
 \[
  \int_{-\infty}^{\infty}\textup{e}^{-t\delta}\lvert f_x(t)\rvert\,\textup{d}t
 \]
 exists.
 \item \label{it:boundedC}\emph{Boundedness of $N$}. There exists $\mathfrak C>0$ such that $\textup{e}^{-t\delta}N^{\text{abs}}(t,x)\leq\mathfrak{C}$ for all $x\in E^{\infty}$ and $t\in\mathbb R$, where
 \begin{equation*}
  N^{\text{abs}}(t,x)\defeq\sum_{n=0}^{\infty}\sum_{y:\sigma^ny=x}\kappa(y)\lvert f_y(t-S_n\xi(y))\rvert\textup{e}^{S_n\eta(y)}.
 \end{equation*}
 \item\label{it:decay} \emph{Exponential decay of $N$ on the negative half-axis}. There exist $\widetilde{\mathfrak C}>0,s>0$ and $t_0\in\mathbb R$ such that $\textup{e}^{-t\delta}N^{\text{abs}}(t,x)\leq \widetilde{\mathfrak C}\textup{e}^{st}$ for all $t\leq t_0$.
\end{enumerate}

\begin{thm}[Renewal theorem, {\cite{KK15a}}]\label{thm:RT1}
   Assume that  $x\mapsto f_x(t)$ is $\theta$-H\"older continuous for every $t\in\mathbb R$ and that Conditions \ref{it:regularPotential} to \ref{it:decay} hold.
   \begin{enumerate}
   \item\label{it:RT1:nl} If $\xi$ is non-lattice and $f_x$ is monotonic for every $x\in E^{\infty}$, then 
   \begin{equation*}
    N(t,x)\sim\textup{e}^{t\delta}\eigenf_{\eta-\delta \xi}(x)\underbrace{
\frac{1}{\int \xi\,\textup{d}\mu_{\eta-\delta \xi}}\int_{E^{\infty}} \kappa(y)\int_{-\infty}^{\infty}\textup{e}^{-T\delta}f_y(T)\,\textup{d}T\,\textup{d}\nu_{\eta-\delta \xi}(y)
}_{\eqdef G}
   \end{equation*}
   as $t\to\infty$, uniformly for $x\in E^{\infty}$.
   \item\label{it:RT1:l} Assume that $\xi$ is lattice and let $\zeta,\psi\in\mathcal{C}(E^{\infty})$ satisfy the relation
   \[
   \xi-\zeta=\psi-\psi\circ\sigma,
   \]
   where $\zeta$ is a function whose range is contained in a discrete subgroup of $\mathbb R$. Let $a>0$ be maximal such that $\zeta(E^{\infty})\subseteq a\mathbb Z$. Then
  \begin{align*}
    N(t,x)
    \sim \textup{e}^{t\delta}\eigenf_{\eta-\delta\zeta}(x)\widetilde{G}_x(t)
  \end{align*}
  as $t\to\infty$, uniformly for $x\in E^{\infty}$, where $\widetilde{G}_x$ is periodic with period $a$ and
  \begin{align*}
    \widetilde{G}_x(t)
    &\defeq \int_{E^{\infty}}\kappa(y)
\sum_{\ell=-\infty}^{\infty}\textup{e}^{-a \ell\delta}f_y\left(a \ell+a\left\{\tfrac{t+\psi(x)}{a}\right\}-\psi(y)\right)\,
    \textup{d}\nu_{\eta-\delta\zeta}(y)\\
    &\qquad\times \textup{e}^{-a\big{\{}\frac{t+\psi(x)}{a}\big{\}}\delta}
    \frac{a\textup{e}^{\delta\psi(x)}}{\int\zeta\,\textup{d}\mu_{\eta-\delta\zeta}}.
  \end{align*}
  \item\label{it:RT1:av}We always have 
  \begin{equation*}
  \lim_{t\to\infty}\frac{1}{t}\int_0^{t}\textup{e}^{T\delta}N(T,x)\,\textup{d}T=G\cdot\eigenf_{\eta-\delta \xi}(x).
  \end{equation*}
  \end{enumerate}
\end{thm}

\begin{rem}\label{rem:monotonicity}
 The monotonicity required in \ref{it:RT1:nl} can be replaced by other conditions. For instance we could require that there exists $n\in\mathbb N$ for which $S_n\xi$ is bounded away from zero and that the family $(t\mapsto\textup{e}^{-t\delta}\lvert f_x(t)\rvert\mid x\in E^{\infty})$ is equi-directly Riemann integrable (cf. \cite{renewal}). \end{rem}

\section{Proofs}\label{sec:proofs}

\begin{lem}[Bounded Distortion]\label{lem:bd}\cite{KK15a}
  There exists a sequence $(\bd_n)_{n\in\mathbb N}$ with $\bd_n>0$ for all $n\in\mathbb N$ and $\lim_{n\to\infty}\bd_n=1$ such that for all $\om,u\in E^*$ with $u\om\in E^*$ and $x,y\in\phi_{\om}(X_{t(\om)})$ we have that
  \begin{equation*}
    \bd_{n(\om)}^{-1}\leq\frac{\lvert\phi_{u}'(x)\rvert}{\lvert\phi_{u}'(y)\rvert}\leq\bd_{n(\om)}.
  \end{equation*}
\end{lem}

\begin{proof}[Proof of Prop.~\ref{prop:non-trivial}]
	This proof is an extension of the proof of  \cite[Prop.~4{.}5{.}9]{1033.37025}.
	Non-triviality implies the existence of $\tilde{ v}\in V$ for which $\lambda_d(\Gamma_{\tilde{v}})>0$. Finite irreducibility implies that for each $v\in V$ there exists $\om_v\in\Lambda$ with $\phi_{\om_v}\Gamma_{\tilde v}\subset X_v$. Let $m\defeq\max\{n(\om)\mid \om\in\Lambda\}$ and for $v\in V$ set
	\[
		G_v\defeq O_v\setminus\bigcup_{\om\in E_A^{m+1}}\phi_{\om}O_{t(\om)}.
	\]
	Then each $G_v$ has positive $d$-dimensional Lebesgue measure, since
	\begin{align*}
		G_v
		&\supseteq O_v\setminus \bigcup_{\om\in E_A^{n(\om_v)+1}}\phi_{\om} O_{t(\om)}
		\supseteq\bigcup_{\om\in E_A^{n(\om_v)}, i(\om)=v} \phi_{\om}O_{t(\om)}\setminus \bigcup_{\om\in E_A^{n(\om_v)+1}}\phi_{\om} O_{t(\om)}\\
		&= \bigcup_{\om\in E_A^{n(\om_v)}, i(\om)=v} \phi_{\om}\left(O_{t(\om)}\setminus \bigcup_{e\in E}\phi_{e} O_{t(e)}\right)
		\supseteq \phi_{\om_v}\left(\Gamma_{\tilde v}\right)
	\end{align*}
	implies $\lambda_d(G_v)\geq \lambda_d(\phi_{\om_v}\Gamma_{\tilde{v}})>0$. Recall that $V$ has finite cardinality. Therefore, $K\defeq \min_{v\in V} \lambda_d(G_v)/\lambda_d(O_v)\bd_1^{-2d}$ exists and lies in the open interval $(0,1)$. By bounded distortion and the transformation formula we have that
	\begin{equation*}\label{eq:bdGO}
		\lambda_d(\phi_{\om}G_{t(\om)})
		=\int_{G_{t(\om)}}\lvert \phi'_{\om}\rvert^d\,\textup{d}\lambda_d 
		\geq K \int_{O_{t(\om)}}\lvert \phi'_{\om}\rvert^d\,\textup{d}\lambda_d 
		=K\cdot \lambda_d(\phi_{\om}O_{t(\om)}).
	\end{equation*}
	For $n\in\mathbb N_0$ write
	\[
		W_n\defeq \bigcup_{\om\in E_A^n}\phi_{\om}O_{t(\om)}.
	\]
	Then $O_v\cap W_{m+1}= O_v\cap \bigcup_{\om\in E_A^{m+1}}\phi_{\om}O_{t(\om)}=O_v\setminus G_v$ and whence
	\begin{align*}
		W_{m+1+n}
		=\bigcup_{\om \in E_A^n}\phi_{\om}(W_{m+1})
		=\bigcup_{\om \in E_A^n}\phi_{\om}(O_{t(\om)}\setminus G_{t(\om)})
		=W_n\setminus \bigcup_{\om \in E_A^n}\phi_{\om}G_{t(\om)}.
	\end{align*}
	This yields
	\begin{align*}
		\lambda_d(W_{m+1+n})
		&=\lambda_d(W_n)-\lambda_d\left(  \bigcup_{\om \in E_A^n}\phi_{\om}G_{t(\om)} \right)
		=\lambda_d(W_n)-\sum_{\om \in E_A^n}\lambda_d\left(\phi_{\om}G_{t(\om)} \right)\\
		&\leq \lambda_d(W_n)-K\sum_{\om \in E_A^n}\lambda_d\left(\phi_{\om}O_{t(\om)} \right)
		=  (1-K)\cdot\lambda_d(W_n).
	\end{align*}
	By the bounded distortion property, $\|\phi'_{\om}\|_{\infty}^d\leq \bd_1^d \lambda_d(\phi_{\om}O)/\lambda_d(O)$. Thus, 
	\begin{align*}
	P(-d\xi)
	&=\lim_{n\to\infty}\frac{1}{n(m+1)}\log\sum_{\om\in E_A^{n(m+1)}}\exp\left(\sup_{\tau\in[\om]}-dS_{n(m+1)}\xi(\tau)\right)\\
	&\leq \lim_{n\to\infty}\frac{1}{n(m+1)}\log\bd_1^d \frac{\lambda_d(W_{n(m+1)})}{\lambda_d(O)}\\
	&\leq \frac{1}{m+1}\log(1-K)<0.
	\end{align*}
	This implies that $\dim_M(F)<d$. 
\end{proof}

\begin{lem}[{\cite[Lem.~4{.}4]{Diss}}]
	The set-class
	\begin{align*}
	\mathcal E_F&\defeq\{\phi_{\om}O\mid\om\in E^*\}\cup\mathcal K_F\quad\text{with}\\
	\mathcal K_F &\defeq\{K\in\mathcal B(\mathbb R^d)\mid\exists n\in\mathbb N\colon K\subset\mathbb R^d\setminus\bigcup_{\om\in E_A^n}\phi_{\om}O\}
	\end{align*}
	forms an intersection stable generator of $\mathcal B(\mathbb R^d)$.
\end{lem}

\begin{proof}[Proof of Thm.~\ref{thm:main}]
Directly from the definition of $\Gamma_v$ in \eqref{eq:Gamma} we obtain that  $\bigcup O_v$ decomposes in the following way.
\begin{align*}
  \bigcup_{v\in V}O_v
  &=\bigcup_{v\in V}\Gamma_v\cup \bigcup_{e\in E:i(e)=v}\phi_e(O_{t(e)}) = \cdots\\
  &=\bigcup_{n=0}^{\infty}\bigcup_{\om\in E_A^n}\phi_{\om}(\Gamma_{t(\om)})\cup  \bigcap_{n=0}^{\infty}\bigcup_{\om\in E_A^n}\phi_{\om}(O_{t(\om)}).
\end{align*}
Since $ \bigcap_{n\in\mathbb N}\bigcup_{\om\in E_A^n}\phi_{\om}(O_{t(\om)})\subseteq \bigcap_{n\in\mathbb N}\bigcup_{\om\in E_A^n}\phi_{\om}(X_{t(\om)}) =F$ its Lebesgue measure is zero by Prop.~\ref{prop:non-trivial} and it follows that
\begin{align}
  \lambda_d\left(F_{\eps}\cap B\right)
  &=\sum_{n=0}^{m-1}\sum_{\om\in E_A^n} \lambda_d\left(F_{\eps}\cap\phi_{\om}(\Gamma_{t(\om)})\cap B\right) \label{eq:lebesgueseries}\\
  &\quad + 
  \sum_{u\in E_A^m}\sum_{n=0}^{\infty}\sum_{\substack{\om\in E_A^n\\ t(\om)=i(u)}} \lambda_d\left(F_{\eps}\cap\phi_{\om u}(\Gamma_{t(u)})\cap B\right)\nonumber
\end{align}
for any fixed $m\in\mathbb N$. We  write $\ee^{-t}=\eps$ and use that by the projection condition \eqref{eq:projcond} we have that $F_{\eps}\cap\phi_{\om}\Gamma_{t(\om)}=\phi_{\om}F_{\eps}\cap\phi_{\om}\Gamma_{t(\om)}$ for $\om\in E^*$. 
We separately consider the two summands  on the right hand side of \eqref{eq:lebesgueseries} and start with the first one. For $\om\in E^*$, we choose an arbitrary $y_\om\in E^{\infty}$ for which $\om y_{\om}\in E^{\infty}$ and apply the Bounded Distortion Lemma:
\begin{align*}
 \lambda_d\left(F_{\ee^{-t}}\cap\phi_{\om}\Gamma_{t(\om)}\right)
 &\stackrel{\ref{it:cond1}}{=}\lambda_d\left(\phi_{\om}F_{\ee^{-t}}\cap\phi_{\om}\Gamma_{t(\om)}\right)\\
 &\leq \bd_1^d \ee^{-dS_n\xi(\om y_{\om})}\lambda_d\left(F_{\ee^{-t+S_n\xi(\om y_{\om})+\ln \bd_1}}\cap\Gamma_{t(\om)}\right)\\
 &\stackrel{\ref{it:cond4}}{\leq} \bd_1^d\ee^{-dS_n\xi(\om y_{\om})}c\ee^{(-t+S_n\xi(\om y_{\om})+\ln \bd_1)(d-D+\gamma)}\\
 &\leq  c\bd_1^{2d-D+\gamma} \ee^{(\gamma-D)S_n\xi(\om y_{\om})} \ee^{-t(d-D+\gamma)}.
\end{align*}
Thus,
\begin{align}
 \ee^{-t(D-d)}\sum_{n=0}^{m-1}\sum_{\om\in E_A^n} \lambda_d\left(F_{\ee^{-t}}\cap\phi_{\om}(\Gamma_{t(\om)})\right)
 &\leq c\bd_1^{2d-D+\gamma} \ee^{-t\gamma}\sum_{n=0}^{m-1}\sum_{\om\in E_A^n}\ee^{(\gamma-D)S_n\xi(\om y_{\om})}\nonumber\\
 & = c\bd_1^{2d-D+\gamma} \ee^{-t\gamma}\sum_{n=0}^{m-1}\mathcal L_{(\gamma-D)\xi}^n\mathds 1(y_{\om})\label{eq:sumtom}\\
 &\stackrel{t\to\infty}{\longrightarrow} 0\nonumber
\end{align}
for fixed $m\in\mathbb N$.
This in particular shows that $\lim_{t\to\infty}\lambda_d(F_{\ee^{-t}}\cap B)=0=\nu(B)$ for $B\in\mathcal K_F$.
Now, we turn to the second summand on the right hand side of \eqref{eq:lebesgueseries} and consider the case that $B=\phi_{\kappa}O\in\mathcal E_F\setminus\mathcal K_F$. For each $u\in E_A^m$ with $m\geq n(\kappa)$, we choose an arbitrary $x_u\in[u]$ and again apply the Bounded Distortion Lemma. 
\begin{align*}
  &\sum_{n=0}^{\infty}\sum_{\substack{\om\in E_A^n\\ t(\om)=i(u)}} \lambda_d\left(F_{\ee^{-t}}\cap\phi_{\om u}(\Gamma_{t(u)})\right)\cdot\mathds 1_{[\kappa]}(\om x_u)\\
  &\qquad =\sum_{n=0}^{\infty}\sum_{\substack{\om\in E_A^n:\\t(\om)=i(u)}} \lambda_d\left((\phi_{\om}F)_{\ee^{-t}}\cap\phi_{\om u}(\Gamma_{t(u)})\right))\cdot\mathds 1_{[\kappa]}(\om x_u)\\
 &\qquad\leq \sum_{n=0}^{\infty} \sum_{y:\sigma^ny=x_u} \bd_m^d\ee^{-dS_n\xi(y)} \lambda_d\left(F_{\ee^{-t+S_n\xi(y)+\ln\bd_m}}\cap\phi_{u}(\Gamma_{t(u)})\right))\cdot\mathds 1_{[\kappa]}(y)\\
 &\qquad = \bd_m^d  \sum_{n=0}^{\infty} \sum_{y:\sigma^ny=x_u} \ee^{-dS_n\xi(y)} f^{u}(t-S_n\xi(y)-\ln\bd_m)\cdot\mathds 1_{[\kappa]}(y),
\end{align*}
where 
\[
 f^{u}(t)\defeq \lambda_d\left(F_{\ee^{-t}}\cap\phi_{u}(\Gamma_{t(u)})\right).
\]
With $\eta\defeq -d\xi$, $\kappa\equiv\mathds 1$, $f_y=f^u$ for any $y$ with $\sigma^ny=x_u$ and $N$ as in \eqref{eq:N} we obtain
\begin{align}
  \sum_{u\in E_A^m}\sum_{n=0}^{\infty}\sum_{\substack{\om\in E_A^n\\ t(\om)=i(u)}} \lambda_d\left(F_{\ee^{-t}}\cap\phi_{\om u}(\Gamma_{t(u)})\right)
 \leq \bd_m^d \sum_{u\in E_A^m} N(t-\ln\bd_m,x_u)\label{eq:proofupper}
\end{align}
and likewise
\begin{align}
  \sum_{u\in E_A^m}\sum_{n=0}^{\infty}\sum_{\substack{\om\in E_A^n\\ t(\om)=i(u)}} \lambda_d\left(F_{\eps}\cap\phi_{\om u}(\Gamma_{t(u)})\right)
 \geq \bd_m^{-d} \sum_{u\in E_A^m} N(t+\ln\bd_m,x_u).\label{eq:prooflower}
\end{align}

We now want to apply the renewal theorem to the right hand side of the above equations and thus check that its assumptions are satisfied. 

\begin{enumerate}
 \item[\emph{Ad} \ref{it:regularPotential}:] By strong regularity $P(-D\xi)=0$ and as $\xi>0$ the map $t\mapsto P(t\xi)$ is increasing on $\mathbb R$ and strictly increasing on $\{t\in\mathbb R\mid P(t\xi)<\infty\}\supseteq(-\infty,\theta]\supsetneq (-\infty,-D]$ with $\theta$ as in \eqref{eq:theta}. Thus $\delta=D-d$ is unique with the property $P(-(d+\delta)\xi)=0$. Moreover, $-\delta<\theta+d=t^*$. On $(-\infty,\theta)$ the function $t\mapsto P(t\xi)$ is differentiable with derivative $\int\xi\,\textup{d}\mu_{t\xi}$. Therefore, $\int (d-t)\xi\,\textup{d}\mu_{-(d+\delta)\xi}=(d-t) \int \xi\,\textup{d}\mu_{-D\xi}$ is finite.
 \item[\emph{Ad} \ref{it:Lebesgue}:] Is satisfied by \ref{it:cond4}.
 \item[\emph{Ad} \ref{it:decay}:] Choose $t_0<0$ so that $F_{\ee^{-t_0}}\cap \phi_u\Gamma_{t(u)}=\phi_u\Gamma_{t(u)}$. Then for $t<t_0$:
 \begin{align*}
  N(t,x)&=\sum_{n=0}^{\infty}\sum_{y:\sigma^n y=x}\lambda_d(\phi_u\Gamma_{t(u)})\ee^{-dS_n\xi(y)}\\
  &=\lambda_d(\phi_u\Gamma_{t(u)})\sum_{n=0}^{\infty}\mathcal L_{-d\xi}^n\mathds 1(x)
  \eqdef \widetilde{C}<\infty
 \end{align*}
 by the spectral radius formula and Cor.~\ref{cor:disc-real}. The assertion now follows with $s\defeq -\mdim$.
 \item[\emph{Ad} \ref{it:boundedC}:] A standard trick (e.\,g.\ used in \cite{Lalley}) is to consider the function 
 $M(t,x)\defeq \ee^{-t\mdim}N(t,x)/\eigenf_{-D\xi}(x)$. By the renewal equation we have
 \begin{align*}
  M(t,x)
  =&\sum_{y:\sigma y=x}M(t-\xi(y),y)\ee^{-D\xi(y)}\frac{\eigenf_{-D\xi}(y)}{\eigenf_{-D\xi}(x)}\\
  &\quad +\ee^{-t\mdim}\lambda_d(F_{\ee^{-t}}\cap \phi_u\Gamma_{t(u)})/\eigenf_{-D\xi}(x).
 \end{align*}
 Let $\overline{M}(t,x)\defeq\sup_{t'\leq t} M(t,x)$, $\overline{M}(t)\defeq\sup_{x\in E^{\infty}}\overline{M}(t,x)$ and recall that $\xi(x)\geq -\ln r>0$ and $\eigenf_{-D\xi}\geq R>0$ (see Thm.~\ref{thm:Ruelleinfinite}\ref{it:R}). Then
 \begin{align*}
  \overline{M}(t,x)
  &\leq \sum_{y:\sigma y=x}\overline{M}(t+\ln r,y)\ee^{-D\xi(y)}\frac{\eigenf_{-D\xi(y)}}{\eigenf_{-D\xi(x)}}\\
  &\quad+\sup_{t'\leq t}\ee^{-t'\mdim}\lambda_d(F_{\ee^{-t'}}\cap \phi_u\Gamma_{t(u)})/\eigenf_{-D\xi}(x)\\
  &\leq \overline{M}(t+\ln r)  + \sup_{t'\leq t}\ee^{-t'\mdim}\lambda_d(F_{\ee^{-t'}}\cap \phi_u\Gamma_{t(u)})/R.
 \end{align*}
  Let $t_0$ be as in the proof of \ref{it:decay}. Then $\overline{M}(t_0)\leq \widetilde{C}\ee^{-\mdim t_0}/R$ and for $t\leq t_0$
  \begin{align*}
   \overline{M}(t)
   \leq \overline{M}(t+\ln r)  + \ee^{-t\mdim}\lambda_d(\phi_u\Gamma_{t(u)})/R,
  \end{align*}
  which for $n\in\mathbb N$ implies 
  \begin{align*}
   \overline{M}(t_0-n\ln r)
   \leq \widetilde{C}\ee^{-\mdim t_0}/R  + \ee^{-t_0\mdim}\lambda_d(\phi_u\Gamma_{t(u)})/(R(1-r^{\mdim}))
   \eqdef B.
  \end{align*}
  Whence $\sup_{t\in\mathbb R} M(t,x)\leq B$ for all $x$. As $\eigenf_{-D\xi}$ is bounded also $\ee^{-\mdim t}N(t,x)$ is bounded.
\end{enumerate}
For applying the Renewal Thm.~\ref{thm:RT1} to \eqref{eq:proofupper} resp.\ \eqref{eq:prooflower} we distinguish between the lattice and non-lattice situations.

If $\xi$ is non-lattice the renewal theorem yields
\begin{align*}
 &\sum_{u\in E_A^m} N(t\pm\ln\bd_m,x_u)\\
 &\sim \sum_{u\in E_A^m}\textup{e}^{t\delta}\bd_m^{\pm\delta}\eigenf_{-(d+\delta) \xi}(x_u)
\frac{1}{\int \xi\,\textup{d}\mu_{-(d+\delta) \xi}}\int_{-\infty}^{\infty} \textup{e}^{-T\delta}f^{u}(T)\,\textup{d}T\cdot\nu_{-(d+\delta)\xi}([\kappa])\\
 &=\textup{e}^{-t(d-D)}\bd_m^{\pm(D-d)}\sum_{u\in E_A^m}\eigenf_{-D\xi}(x_u)
\frac{1}{\int \xi\,\textup{d}\mu_{-D\xi}}\int_{-\infty}^{\infty} \textup{e}^{-T(D-d)}f^{u}(T)\,\textup{d}T\cdot\nu_{-D\xi}([\kappa])
\end{align*}
Combining \eqref{eq:lebesgueseries} -- \eqref{eq:prooflower} with the above we obtain that for $B=\phi_{\kappa}O$
\begin{align*}
 &\lim_{m\to\infty}\lim_{t\to\infty}\ee^{-t(D-d)}\lambda_d(F_{\ee^{-t}}\cap B)\\
 &\quad = \frac{1}{\int\xi\,\textup{d}\mu_{-D\xi}}\lim_{m\to\infty}\sum_{u\in E_A^m}\eigenf_{-D\xi}(x_u)\int_{-\infty}^{\infty}\ee^{-T(D-d)}\lambda_d(F_{\ee^{-T}}\cap\phi_u\Gamma_{t(u)})\,\textup{d}T \cdot\nu(B).
\end{align*}
Eq.~\eqref{eq:convergencetoef} implies $\eigenf_{-D\xi}(x_u)=\lim_{n\to\infty}\mathcal L^n_{-D\xi}\mathds 1(x_u)=\lim_{n\to\infty}\sum_{\om\in E_A^n}\lvert \phi'_\om(\pi x_u)\rvert^D$ which gives
\begin{align*}
 &\lim_{m\to\infty}\lim_{t\to\infty}\ee^{-t(D-d)}\lambda_d(F_{\ee^{-t}}\cap B)\\
 &\quad = \frac{1}{\int\xi\,\textup{d}\mu_{-D\xi}}\lim_{m\to\infty}\sum_{u\in E_A^m}\int_{-\infty}^{\infty}\ee^{-T(D-d)}\lambda_d(F_{\ee^{-T}}\cap\phi_u\Gamma_{t(u)})\,\textup{d}T \cdot\nu(B).
\end{align*}

If $\xi$ is lattice then
\begin{align*}
 &\sum_{u\in E_A^m} N(t\pm\ln\bd_m,x_u)\\
 &\sim \sum_{u\in E_A^m} \ee^{t\mdim} \bd_m^{\pm\mdim}\eigenf_{-d\xi-\mdim\zeta}(x_u) \ee^{-\aaa\mdim\left\{\frac{t\pm\ln\bd_m+\psi(x_u)}{\aaa}\right\}} \frac{\aaa\ee^{\mdim\psi(x_u)}}{\int\zeta\,\textup{d}\mu_{-d\xi-\mdim\zeta}}\\
 &\quad\times\int_{[\kappa]} \sum_{\ell=-\infty}^{\infty}\ee^{-\aaa\ell\mdim} \lambda_d\left(F_{\ee^{-(\aaa\ell+\aaa\left\{\frac{t\pm\ln\bd_m+\psi(x_u)}{\aaa}\right\}-\psi(y))}}\cap\phi_u\Gamma_{t(u)}\right)\,\textup{d}\nu_{-d\xi-\mdim\zeta}(y) \\
 &\eqdef E_m(t)
\end{align*}
Eq.~\eqref{eq:convergencetoef} implies that $\eigenf_{-d\xi-\delta\z}(x_u)=\lim_{n\to\infty}\sum_{\om\in E_A^n, \om u\in E^*}\lvert\phi'_{\om}(\pi x_u)\rvert^d\ee^{-\delta S_n\z(\om x_u)}$. Using this and that $S_n\z(E^{\infty})\subset\aaa\mathbb Z$ we obtain
\begin{align*}
 E_m(t)
 \leq&\frac{\aaa\bd_m^d}{\int\z\,\textup{d}\mu_{-d\xi-\delta\z}}\sum_{u\in E_A^m}\lim_{n\to\infty}\sum_{\om\in E_A^n,\om u\in E^*}\\
 & \times\int_{[\kappa]}\sum_{\ell=-\infty}^{\infty}\ee^{-\aaa\ell\mdim}\lambda_d\left(F_{\ee^{-(\aaa\ell+t\pm\ln\bd_m-\psi(y)+\psi(\om x_u))}\bd_m}\cap\phi_{\om u}\Gamma_{t(u)}\right)\,\textup{d}\nu_{-d\xi-\delta\z}(y).
\end{align*}
A lower bound can be found analoguously, yielding
\begin{align*}
 &\lim_{m\to\infty} E_m(t)\\
 &=\ee^{\mdim t}\frac{\aaa}{\int\z\,\textup{d}\mu_{-d\xi-\delta\z}}
 \lim_{m\to\infty}\sum_{u\in E_A^m}\ee^{\delta\psi(x_u)}\sum_{\ell=-\infty}^{\infty}\ee^{-\aaa\ell\mdim}\\
 &\times\int_{[\kappa]}\ee^{-\aaa\mdim\left\{\frac{t+\psi(x_u)-\psi(y)}{\aaa}\right\}-\mdim\psi(y)}\lambda_d\left(F_{\ee^{-\aaa(\ell+\left\{\frac{t+\psi(x_u)-\psi(y)}{\aaa}\right\})}}\cap\phi_{\om u}\Gamma_{t(u)}\right)\,\textup{d}\nu_{-d\xi-\delta\z}(y).
\end{align*}
With $\textup{d}\nu_{-d\xi-\mdim\z}=\ee^{\mdim\psi}\textup{d}\nu_{-(d+\mdim)\xi}$, $\int\z\,\textup{d}\mu_{-d\xi-\mdim\z}=\int\xi\,\textup{d}\mu_{-(d+\delta)\xi}$ and $\mu_{-d\xi-\mdim\z}=\mu_{-(d+\mdim)\xi}$ the statement follows.
\end{proof}

\begin{proof}[Proof of Thm.~\ref{thm:image}]
	Let $\Phi\defeq\left(\phi_e\colon X_{t(e)}\to X_{i(e)}\right)_{e\in E}$ denote the cGDS associated with $F_0$. Then  $\Psi\defeq\left(\psi_e\defeq g\phi_e g^{-1}\colon g\left(X_{t(e)}\right)\to g\left(X_{i(e)}\right)\right)_{e\in E}$ is a cGDS with invariant set $F$. Let $\xi$ and $\widetilde{\xi}$ denote the geometric potential functions associated with $\Phi$ and $\Psi$, respectively. We have $\mu_{-D\widetilde{\xi}}=\mu_{-D\xi}$ and since $\mu_{-D\xi}$ is $\sigma$-invariant, also
	\begin{equation}\label{eq:entropy}
		\int\widetilde{\xi}\,\textup{d}\mu_{-D\widetilde{\xi}}
		=\int\xi\,\textup{d}\mu_{-D\xi}
	\end{equation}
	Moreover, for $n<m$ and an arbitrary $x\in E^{\infty}$ with $ux\in E^{\infty}$ we have
	\begin{align*}
		&\sum_{n(\om)=m}\sum_{k=1}^{\infty}\left\lvert \psi_{\om} g  f^k_{t(\om)}(C_4) \right\rvert^D\\
		&\quad=\sum_{n(\om)=m}\sum_{k=1}^{\infty}\left\lvert g \phi_{\om} f^k_{t(\om)}(C_4) \right\rvert^D\\
		&\quad=\sum_{n(u)=n}\sum_{\substack{n(\om)=m-n\\ \om u\in E^*}}\sum_{k=1}^{\infty}\left\lvert g \phi_{\om}\phi_u f^k_{t(u)}(C_4) \right\rvert^D\\
		&\quad\leq\sum_{n(u)=n}\sum_{k=1}^{\infty}\bd_n^{2D}\sum_{\substack{n(\om)=m-n\\ \om u\in E^*}}\left\lvert g'( \phi_{\om}\pi ux)\right\rvert^D \left\lvert\phi'_{\om}(\pi ux)\right\rvert^D \left\lvert \phi_u f^k_{t(u)}(C_4) \right\rvert^D\\
		&\quad=\sum_{n(u)=n}\sum_{k=1}^{\infty}\bd_n^{2D} \mathcal L^{m-n}_{-D\xi} \lvert g'\circ\pi \rvert^D(ux) \left\lvert \phi_u f^k_{t(u)}(C_4) \right\rvert^D		
	\end{align*}
	Applying \eqref{eq:convergencetoef} to the above expression for $m\to\infty$ and using that the lower bound can be found analogously yields
	\begin{align*}
		&\lim_{m\to\infty}\sum_{n(\om)=m}\sum_{k=1}^{\infty}\left\lvert \psi_{\om} g  f^k_{t(\om)}(C_4) \right\rvert^D\\
		&\quad = \lim_{n\to\infty}\sum_{n(u)=n}\sum_{k=1}^{\infty}\left\lvert \phi_u f^k_{t(u)}(C_4) \right\rvert^D\int\lvert g'\circ\pi\rvert^D\,\textup{d}\nu_{-D\xi}.
	\end{align*}
	This equality in particular holds for $g=\id$, whence 
	\begin{equation}
		\frac{\lim_{m\to\infty}\sum_{n(\om)=m}\sum_{k=1}^{\infty}\left\lvert \psi_{\om} g  f^k_{t(\om)}(C_4) \right\rvert^D}{\lim_{m\to\infty}\sum_{n(\om)=m}\sum_{k=1}^{\infty}\left\lvert \phi_{\om}  f^k_{t(\om)}(C_4) \right\rvert^D}
		=\int\lvert g'\circ\pi\rvert^D\,\textup{d}\nu_{-D\xi}.
	\end{equation}
\end{proof}

\begin{proof}[Proof of Thm.~\ref{thm:Euler}]
	With the notation of Sec.~\ref{sec:Oh} we have
	\begin{align*}
	\chi\left(F_{\ee^{-t}}\cap\mathcal T\right) 
	&= 1-R(\ee^{-t})\\
	&= -\sum_{k=1}^{\infty}\sum_{n=0}^{\infty}\sum_{\om\in E_A^n} \mathds 1_{(\ee^{-t},\infty)}\left(\lvert \phi_{\omega}f_{t(\omega)}^k(C_4) \rvert/2\right)\\
	&= -\sum_{k=1}^{\infty}\sum_{n=0}^{m-1}\sum_{\om\in E_A^n} \mathds 1_{(\ee^{-t},\infty)}\left(\lvert \phi_{\omega}f_{t(\omega)}^k(C_4) \rvert/2\right)\\
	&\qquad-\sum_{k=1}^{\infty}\sum_{u\in E_A^m}\sum_{n=0}^{\infty}\sum_{\om\in E_A^n} \mathds 1_{(\ee^{-t},\infty)}\left(\lvert \phi_{\omega}\phi_u f_{t(u)}^k(C_4) \rvert/2\right)
	\end{align*}
	for arbitrary $m\in\mathbb N$. Using the Bounded Distortion Lemma (Lem.~\ref{lem:bd}) and Lem.~\ref{lem:k-2} we obtain that the first series is 
	$\mathcal{O}(t^{m} \ee^{t/2})$ as $t\to\infty$ with $\mathcal{O}$ denoting the Big-O Landau symbol. This can be seen as follows. 
	\begin{align*}
	 &\sum_{n=0}^{m-1}\sum_{\om\in E_A^n} \sum_{k=1}^{\infty}\mathds 1_{(\ee^{-t},\infty)}\left(\lvert \phi_{\omega}f_{t(\omega)}^k(C_4) \rvert/2\right)\\
	 &\quad\leq \sum_{n=0}^{m-1}\sum_{j_1,\ldots,j_n\in\mathbb N}{\sum_{k=1}^{\infty}}3^n\mathds 1_{(\ee^{-t},\infty)}\left(Q^{n+1}\lvert C_4\rvert/2\cdot\left(j_1\cdots j_n k\right)^{-2}\right)\\
	 &\quad\leq \sum_{n=0}^{m-1}3^n\sum_{j_1,\ldots,j_n=1}^{\left\lfloor\ee^{t/2}\sqrt{Q^{n+1}\lvert C_4\rvert/2}\right\rfloor} \ee^{t/2}\sqrt{Q^{n+1}\lvert C_4\rvert/2}\cdot\left(j_1\cdots j_n \right)^{-1}\\
	 &\quad\leq \ee^{t/2} C_1\sum_{n=0}^{m-1}(3Q)^n (C_2 t)^n
	\end{align*}
  with some constants $C_1,C_2>0$ for sufficiently large $t$.
	With the second series we proceed as in the proof of Thm.~\ref{thm:Apollo}. Again using the Bounded Distortion Lemma and setting 
	\[
	g^{u,k}(t)\defeq \mathds 1 _{\left(-\ln \lvert\bd_m\phi_u f_{t(u)}^k(C_4)\rvert/2,\infty\right)}(t)
	\]
 we obtain with an arbitrary $x_u\in[u]$
	\begin{align*}
		\sum_{n=0}^{\infty}\sum_{\om\in E_A^n} \mathds 1_{(\ee^{-t},\infty)}\left(\lvert \phi_{\omega}\phi_u f_{t(u)}^k(C_4) \rvert/2\right)
		&\leq\sum_{n=0}^{\infty}\sum_{\om\in E_A^n} g^{u,k}(t-S_n\xi(\om x_u))\\
		&\sim\ee^{tD}\frac{\eigenf_{-D\xi(x_u)}}{D\int\xi\,\textup{d}\mu_{-D\xi}}\bd_m^D2^{-D}\lvert\phi_{\om} f_{t(\om)}^k(C_4) \rvert^D
	\end{align*}
	as $t\to\infty$. The last asymptotic is a consequence of  Thm.~\ref{thm:RT1}. Its prerequisites are easily checked:
	\begin{enumerate}
		\item[\emph{Ad} \ref{it:Lebesgue}:] 
		\[
		\int_{-\infty}^{\infty} \ee^{-tD}\lvert g^{u,k}(t)\rvert\,\textup{d}t 
		= \int_{-\ln\lvert \phi_u f_{t(u)}^k(C_4) \rvert/2}^{\infty} \ee^{-tD}\,\textup{d}t
		<\infty
		\]
		 \item[\emph{Ad} \ref{it:boundedC}:] This is shown in \cite{Lalley}, since in the present setting $N$ is a counting function.
		\item[\emph{Ad} \ref{it:decay}:]
		$N^{\text{abs}}(t,x)=0$ for $t\leq-\ln\lvert \phi_u f_{t(u)}^k(C_4) \rvert/2$.
	\end{enumerate}
	Using that the Bounded Distortion Lemma provides a lower estimate, too, and applying the same approximation arguments as in the proof of Thm.~\ref{thm:main}, the statement of Thm.~\ref{thm:Apollo} follows.
 \end{proof}

\begin{proof}[Proof of Thm.~\ref{thm:Appoconst}]
  The M\"obius transform $q$ is constructed in such a way that $\{f_3^k q(C_4)\mid k\in\mathbb N\}$ gives a cover of $F_0\cap X_3\cap\{z\in\mathbb C\mid\Re(z)\geq 0\}$. Hence, symmetry implies	
  \begin{align*}
    \mathcal H^D(F_0)
    &\leq 6\lim_{m\to\infty}\sum_{\om\in E_A^m,t(\om)=3}\sum_{k=1}^{\infty}\left|\phi_{\om} f_{3}^k q\left(C_{4}\right)\right|^{D}\\
    &\leq 6\bd_0^{D} \lim_{m\to\infty}\sum_{\om\in E_A^m,t(\om)=3}\sum_{k=1}^{\infty}\left|\phi_{\om} f_{3}^k \left(C_{4}\right)\right|^{D}\\
    &= 2\bd_0^{D} \lim_{m\to\infty}\sum_{\om\in E_A^m}\sum_{k=1}^{\infty}\left|\phi_{\om} f_{{t(\om)}}^k \left(C_{4}\right)\right|^{D}.
  \end{align*}
  Together with Cor.~\ref{cor:Apolloconstant} this proves the first inequality of Thm.~\ref{thm:Appoconst}.
  For obtaining a numeric bound on $c_A$ we need to determine the bounded distortion constant $\bd_0$.
  For this define $\widetilde{\Sigma}(n)\defeq\{(\om_1,\ldots,\om_n)\in\{1,2,3\}^n\mid \om_i\neq\om_{i+1}\ \text{for}\ i\leq n-1\}$. Any allowed concatenation of $\phi_{e_{v,w}^k}$ has a unique representation of the form $\psi_{\om}^k\defeq f_{\om_1}^{k_1}\circ\cdots\circ f_{\om_n}^{k_n}$ for some $n\in\mathbb N$, $\om=(\om_1,\ldots,\om_n)\in\widetilde{\Sigma}(n)$ and $k=(k_1,\ldots,k_n)\in\mathbb N^n$. Moreover, each $\psi_{\om}^k$ is a M\"obius transform and we write $\psi_{\om}^k(z)=(az+b)/(cz+d)$ with $a=a(\om,k), b=b(\om,k), c=c(\om,k), d=d(\om,k)\in\mathbb C$ and $ad-bc\neq 0$.
  Let $v=v(\om,k)$ denote the terminal vertex of $\psi_{\om}^k$. Then
  \begin{align*}
   \bd_0
   =\sup_{\om,k}\frac{\max_{z\in X_v}\lvert(\psi_{\om}^k)'(z)\rvert}{\min_{z\in X_v}\lvert(\psi_{\om}^k)'(z)\vert}
   =\sup_{\om,k}\frac{\max_{z\in X_v}\lvert cz+d\rvert^2}{\min_{z\in X_v}\lvert cz+d\vert^2}.
  \end{align*}
  If $c=0$ the quotient is minimal. Hence, we can assume that $c\neq 0$ and obtain
  \begin{align*}
   \bd_0
   =\sup_{\om,k}\frac{\max_{z\in X_v}\lvert z+d/c\rvert^2}{\min_{z\in X_v}\lvert z+d/c\vert^2}
   =\sup_{\om,k}\frac{(\lvert S_v+d/c\rvert+R)^2}{(\lvert S_v+d/c\rvert-R)^2},
  \end{align*} 
  where $S_v$ is the centre of $X_v$ and $R=2\sqrt{3}-3$ its radius. The last equality holds, since existence of the bounded distortion constant (Lem.~\ref{lem:bd}) implies that $-d/c$ lies in the exterior of $X_v$.
  In the following we show by induction that 
  \begin{align*}
    \hat{\bd}_n
    \defeq \inf_{(\om,k)\in\widetilde{\Sigma}(n)\times \mathbb N^n}
    \left\lvert S_{v(\om,k)}+\frac{d(\om,k)}{c(\om,k)}\right\rvert
    \geq\sqrt{33-18\sqrt{3}}
   \eqdef \hat{\bd}\quad\forall n\in\mathbb N,
  \end{align*}
  yielding
  \[
   \bd_0\leq\left(\frac{\hat{\bd}+R}{\hat{\bd}-R}\right)^2 \leq  4{.}19225.
  \]
  Note that 
  $f_3(z)=\left((\sqrt{3}-1)z+1\right)/\left(-z+\sqrt{3}+1\right)$ and that
  \[
   f_3=g^{-1}\circ h\circ g\quad\text{with}\quad
   g(z)=\frac{1}{z-1},\quad
   g^{-1}(z)=1+\frac{1}{z}\quad\text{and}\quad
   h(z)=z-\frac{\sqrt{3}}{3},
  \]
  see \cite{MR1623671}.
  Thus, $h^k(z)=z-k/\sqrt{3}$ and
  \[
   f_3^k(z)=\frac{(\sqrt{3}-k)z+k}{-kz+k+\sqrt{3}}.
  \]
  Because of symmetry
  \begin{align*}
   \hat{\bd}_1
   =\inf_{k\in\mathbb N}\left\lvert S_1 +\frac{d(3,k)}{c(3,k)} \right\rvert
   =\inf_{k\in\mathbb N}\left\lvert -2+\sqrt{3}+(2\sqrt{3}-3)\mathbf{i} - \frac{k+\sqrt{3}}{k} \right\rvert
   =\hat{\bd}.
  \end{align*}
  Now, take an arbitrary concatenation with representation of the form $\psi_{\om}^k$ with $(\om,k)\in\widetilde{\Sigma}(n+1)\times\mathbb N^{n+1}$. Without loss of generality assume that $\om_{n+1}=3$ and that the terminal vertex is 1. By induction hypothesis we know that $\lvert S_3+\widetilde{d}/\widetilde{c} \rvert\geq \hat{\bd}$, where $\widetilde{c}\defeq c((\om_1,\ldots,\om_n),(k_1,\ldots,k_n))$, $\widetilde{d}\defeq d((\om_1,\ldots,\om_n),(k_1,\ldots,k_n))$, yielding $\widetilde{d}/\widetilde{c}=r\ee^{\mathbf i\theta}-S_3$ for some $r\geq \hat{\bd}$ and $\theta\in[0,2\pi)$.
  Multiplying the associated matrices of the M\"obius maps we see that $c(\om,k)=(\sqrt{3}-k_{n+1})\widetilde{c}-k_{n+1}\widetilde{d}$ and $d(\om,k)=k_{n+1}\widetilde{c}+(k_{n+1}+\sqrt{3})\widetilde{d}$, whence 
  \begin{equation}\label{eq:quot}
   \frac{d(\om,k)}{c(\om,k)}
   =\underbrace{\frac{k_{n+1}}{\sqrt{3}-k_{n+1}}}_{\eqdef A}+\underbrace{\frac{3}{\sqrt{3}-k_{n+1}}\cdot\frac{r\ee^{\mathbf{i}\theta}-S_3}{\sqrt{3}-k_{n+1}-k_{n+1}(r\ee^{\mathbf i\theta}-S_3)}}_{\eqdef B(\theta,r)}.
  \end{equation}
  For $k_{n+1}\geq 2$ the angle between $S_1+A$ and $B(\theta,r)$ is acute, as the scalar product of the two vectors is positive for any $\theta\in[0,2\pi)$ and $r\geq\hat{\bd}$. (Recall that $S_1=-2+\sqrt{3}+(2\sqrt{3}-3)\mathbf{i}$.) Therefore, $\lvert S_1+A+B(\theta,r) \rvert\geq\lvert S_1+A\rvert$, and whence
  \[
   \inf_{\substack{(\om,k)\in\widetilde{\Sigma}(n+1)\times\mathbb N^{n+1}\\k_{n+1}\geq 2}} \left\lvert S_1+\frac{d(\om,k)}{c(\om,k)}\right\rvert
   \geq \inf_{m\geq 2} \left\lvert S_1+\frac{m}{\sqrt{3}-m}\right\rvert
   =\hat{\bd}.
  \]
  Finally, if $k_{n+1}=1$ then \eqref{eq:quot} gives that $d(\om,k)/c(\om,k)>\hat{\bd}$ for any $\theta$ and $r$. Thus, $\hat{\bd}_n\geq\hat{\bd}$ for all $n\in\mathbb N$. 

  In \cite{MR2941628} the Lyapunov-exponent $\int\xi\,\textup{d}\mu_{-D\xi}$ has been computed, where an approximate value of $0{.}9149$ was obtained. All in all, a lower bound for $c_A$ is
  \[
   c_A\geq \frac{\pi^{D/2}}{D2^{D+1}}\cdot\frac{\bd_0^{-D}}{0{.}915}
   \geq 0{.}055.
  \]
\end{proof}

\bibliographystyle{alpha}
\bibliography{Apollonian}

\end{document}